\documentclass[a4paper]{article}

\usepackage[utf8]{inputenc}
\usepackage[width=15cm,height=22cm]{geometry}
\usepackage{amsmath,amsfonts,amssymb,amsthm,mathrsfs}
\usepackage[hidelinks,pdfusetitle]{hyperref}
\usepackage[capitalise]{cleveref}
\usepackage[dvipsnames]{pstricks}
\usepackage{bbm}
\usepackage{graphicx}
\usepackage{stmaryrd} 
\usepackage{tikz}
\usepackage{enumitem}
\crefname{equation}{}{}

\newtheorem{proposition}{Proposition}[section]
\newtheorem{theorem}[proposition]{Theorem}

\newtheorem{definition}[proposition]{Definition}
\newtheorem{corollary}[proposition]{Corollary}
\newtheorem{lemma}[proposition]{Lemma}
\newtheorem{remark}[proposition]{Remark}
\numberwithin{equation}{section}

\newcommand{\dd}{\,\mathrm{d} }
\newcommand{\ad}{\mathrm{ad}}
\newcommand{\D}{\mathrm{d}}

\newcommand{\sgn}{\mathrm{sgn}}
\newcommand{\Lucas}[2]{\begin{bmatrix} #1 \\ #2 \end{bmatrix}}

\begin{document}

\begin{center}
    \uppercase{\textbf{Integer quadratic obstructions to the local controllability of the Burgers equation}}

    \vspace{\baselineskip}

    \large{\textsc{Thomas Perrin\footnote{\textit{Université de Rennes, ENS Rennes, INRIA, CNRS, IRMAR - UMR 6625, F-35000 Rennes, France.}}}}
\end{center}

\begin{abstract}
    We consider the Burgers equation with a scalar control acting through a fixed spatial profile, in a regime where the linearized system is not controllable. We identify natural conditions on the control profile leading to quadratic obstructions to small-time local null-controllability, quantified by negative Sobolev norms of the control of arbitrary integer order. We show that these conditions are related to iterated Lie brackets. The proof relies on repeated integrations by parts in the quadratic expansion combined with nonlinear remainder estimates. Finally, we construct smooth spatial profiles that produce obstructions at every negative integer order.
\end{abstract}

\begingroup
\renewcommand{\thefootnote}{}%
\footnotetext{%
    
    \hspace{0.1cm} \textit{MSC2020:} 93B05, 93C20, 35K58, 35K10, 35Q93
}%
\endgroup

\setcounter{tocdepth}{1}
\tableofcontents

\hspace{2cm}

\section{Introduction}

In this article, we study the local null-controllability of the viscous Burgers equation with scalar control
\begin{equation}\label{eq:Burgers_intro}
    \left\{
    \begin{array}{lll}
        \partial_t y - \partial_x^2 y + y \partial_x y = u(t) \mu(x)
        & \quad t \in (0, T), & \quad x \in (0,1), \\
        y(t,x) = 0
        & \quad t \in (0, T),  & \quad x \in \{0, 1\},\\
        y(0,x)=y_0(x)
        & & \quad x \in (0, 1),
    \end{array}
    \right.
\end{equation}
where $T>0$, $\mu$ is a fixed function (which we sometimes call the \emph{source profile}), $u$ is the control, $y_0 \in L^2(0,1)$ is the initial state, and $y$ is the unknown. Once a source profile $\mu$ is fixed, we write $t \mapsto y(t; y_0, u)$ for the solution of the Burgers equation \eqref{eq:Burgers_intro} with initial data $y_0$ and control $u$. 

In the companion paper \cite{BurgersFractional}, we prove that obstructions to local null-controllability quantified by the $H^{r}$-norm of the control can occur, with $r$ varying across an interval. The purpose of the present article is to show that obstructions may also be quantified by negative Sobolev norms of any integer order. An accessible version of our main result is the following theorem.

\begin{theorem}
    For every integer $N \geq 2$ and every prescribed Fourier direction, one can construct a smooth control profile $\mu$ such that this direction is lost at the linear level, all lower-order quadratic drifts vanish, and a nonzero quadratic drift appears at order $N$. This drift yields an obstruction to small-time local null-controllability with controls small in $H^{2N-3}(0,T)$, quantified by the $H^{-N}$-norm of the control.
\end{theorem}

\subsection{Functional setting}

We denote by $\langle \cdot, \cdot \rangle$ the $L^2(0, 1)$-scalar product. Write $Ay := - \frac{\mathrm{d}^2}{\mathrm{d} x^2} y$ for the Dirichlet Laplacian on $(0,1)$, of domain $D(A) = H^2 \cap H_0^1 (0,1) \subset L^2((0,1); \mathbb{R})$. For $j \in \mathbb{N}^\ast$, write $\lambda_j = ( j\pi )^2$ for the $j$-th eigenvalue of $A$, and $\varphi_j(x) = \sqrt{2} \sin(j \pi x)$ for the corresponding $L^2(0,1)$-orthonormal eigenfunctions. For $n \in \mathbb{N}^\ast$, set $H_{(0)}^n(0,1) = D(A^{\frac{n}{2}})$, with 
\begin{equation}
    \Vert \mu \Vert_{H_{(0)}^n}^2 := \sum_{j \geq 1} \lambda_j^n \langle \mu, \varphi_j \rangle^2.    
\end{equation}
Recall that for $n \in \mathbb{N}$ and $\mu \in H_{(0)}^{2n + 1}(0,1)$, one has $\mu^{(2m)}(0) = \mu^{(2m)}(1) = 0$ for all $m \in \llbracket 0, n \rrbracket$. Finally, set $D(A^\infty) = \bigcap_{n \geq 0} D(A^{\frac{n}{2}})$.

For $u \in L^2(0, T)$ and $\mu \in H^{-1}(0,1)$, the solution of \eqref{eq:Burgers_intro} is well-defined (see \cite{BurgersFractional}), and satisfies 
\begin{equation}
    y \in L^2((0,T);H_0^1(0,1)) \ \cap \ C^0([0,T];L^2(0,1)) \ \cap \ H^1((0,T);H^{-1}(0,1)) .  
\end{equation}
In this article, $\mu$ is assumed to be more regular, and other well-posedness results may also be used (see, for instance, \cite{ShririkyanBurgers,MarbachTimeIteration}). 

For $n \geq 0$ and $T>0$, the space $H^n(0,T)$ denotes the usual Sobolev spaces of real-valued functions. Negative Sobolev regularity of the control will be measured through its iterated primitives, according to the following convention. The quantity $\Vert u_n \Vert_{L^2(0,T)}$ plays the role of the $H^{-n}(0,T)$-norm of $u$ (see, for instance, \cite{BeauchardMarbach2018}).

\begin{definition}[Iterated primitives of the control]\label{def:iterated_primitive}
    Let $T > 0$ and $u \in L^2(0,T)$. For $t \in [0, T]$, we write $u_1(t) := \int_0^t u(\tau) \dd \tau$, and we define recursively $u_{n+1}(t) := \int_0^t u_n(\tau) \dd \tau$ for $n \geq 1$.
\end{definition}

\subsection{Power series expansion}

Assume that $y_0 = 0$. We consider the power series expansion $y = y_1 + y_2 + \cdots$, where $y_1$ and $y_2$ solve, respectively, the linearized and quadratic control problem 
\begin{equation}\label{eq:Burgers_y_1_intro}
    \left\{
    \begin{array}{lll}
        \partial_t y_1 - \partial_x^2 y_1 = u(t) \mu
        & \quad t \in (0, T), & \quad x \in (0, 1), \\
        y_1(t,x) = 0
        & \quad t \in (0, T),  & \quad x \in \{0, 1\},\\
        y_1(0,x)=0
        & & \quad x \in (0, 1),
    \end{array}
    \right.
\end{equation}
and
\begin{equation}\label{eq:Burgers_y_2_intro}
    \left\{
    \begin{array}{lll}
        \partial_t y_2 - \partial_x^2 y_2 + y_1 \partial_x y_1 = 0
        & \quad t \in (0, T), & \quad x \in (0, 1), \\
        y_2(t,x) = 0
        & \quad t \in (0, T), & \quad x \in \{0, 1\},\\
        y_2(0,x)=0
        & & \quad x \in (0, 1).
    \end{array}
    \right.
\end{equation}
Once a source profile $\mu$ is fixed, we sometimes write $t \mapsto y_1(t; u)$ and $t \mapsto y_2(t; u)$ for the solutions of \eqref{eq:Burgers_y_1_intro} and \eqref{eq:Burgers_y_2_intro} with control $u$. Note that $y_1$ and $y_2$ are given by 
\begin{equation}\label{eq:y_1_Fourier_intro}
    y_1(t,x) = \sum_{j \geq 1} \left( \int_0^t e^{-\lambda_j (t-\tau)} u(\tau) \dd \tau \right) \left\langle \mu, \varphi_j \right\rangle \varphi_j(x)
\end{equation}
and
\begin{equation}\label{eq:y_2_Fourier_intro}
    y_2(t,x) = \frac{1}{2} \sum_{j \geq 1} \left( \int_0^t e^{-\lambda_j (t-\tau)} \left\langle (y_1(\tau))^2, \varphi_j^\prime \right\rangle \dd \tau \right) \varphi_j(x).
\end{equation}

If $\langle \mu, \varphi_k \rangle = 0$ for some $k \geq 1$, then $\langle y_1(t; u), \varphi_k \rangle = 0$ for any control $u$. In this case, we say that the direction $\varphi_k$ is \textit{lost at the linear level}, and one has
\begin{equation}
    \langle y(t; 0, u), \varphi_k \rangle = \langle y_2(t; u), \varphi_k \rangle + \text{ remainder terms,}
\end{equation}
so that the behavior of the controlled system in the direction $\varphi_k$ is governed by the quadratic expansion. 

\subsection{Controllability notions}

\begin{definition}\label{def:intro_STLC}
    Let $\sigma \in \mathbb{N}$. We say that \eqref{eq:Burgers_intro} is \emph{small-time locally null-controllable} (in short, \emph{STLC}), with controls small in $H^{\sigma}$, if the following property holds: for any $T, \eta > 0$, there exists $\delta > 0$ such that, for any $y_0 \in L^2(0,1)$ satisfying $\Vert y_0 \Vert_{L^2} \leq \delta$, there exists a control $u \in H^\sigma(0,T)$ satisfying $\Vert u \Vert_{H^{\sigma}} \leq \eta$, such that the corresponding solution satisfies $y(T; y_0, u) = 0$.
\end{definition}

Let $n, \sigma \in \mathbb{N}$, and assume that a direction  $\varphi_k$ is lost at the linear level. We say that there is an \emph{obstruction to STLC with controls small in $H^{\sigma}$, quantified by the $H^{-n}$-norm of the control}, if STLC with controls small in $H^{\sigma}$ is denied by a systematic small-time drift estimate involving the $H^{-n}$-norm of the control, that is, an estimate of the form
\begin{equation}
    \left\langle y(T;y_0,u), \varphi_k \right\rangle = \left\langle y(T; y_0, 0), \varphi_k \right\rangle + \alpha_k \left\Vert u_n \right\Vert_{L^2(0,T)}^2 + \text{ remainder terms},
\end{equation}
for some $\alpha_k \neq 0$, whenever $T$, $\left\Vert u \right\Vert_{H^\sigma(0,T)}$, and $y_0$ are sufficiently small. 

\begin{remark}
    If $\langle \mu, \varphi_n \rangle \neq 0$ for every $n \geq 1$, and if the sequence $\left( \langle \mu, \varphi_n \rangle \right)_n$ does not decay too fast, then one can prove that STLC holds with controls small in $L^2(0,T)$; see \cite{LiuTakahashiTucsnak2013,MarbachTimeIteration}. In that setting, STLC follows from the exact controllability of the linearized system. By contrast, throughout this article, we assume that a direction is lost at the linear level in order to investigate the resulting quadratic effects.
\end{remark}
 
\subsection{Detailed results}\label{sec:integer_drift_any_order}

Write $\Lucas{n}{p}$ for the $(n,p)$-coefficient of Lucas's triangle:
\begin{equation}
    \begin{array}{lcccccccc}
        1 \\
        1 & 2 \\
        1 & 3 & 2 \\
        1 & 4 & 5 & 2 \\
        1 & 5 & 9 & 7 & 2 \\
        1 & 6 & 14 & 16 & 9 & 2 \\
        1 & 7 & 20 & 30 & 25 & 11 & 2 
    \end{array}
\end{equation}
which is a variant of Pascal's triangle with the number $2$ on the diagonal. More precisely, we set $\Lucas{n}{0} = 1$ for all $n \geq 0$, and 
\begin{equation}
    \Lucas{n}{p} = \binom{n}{p} + \binom{n - 1}{p - 1}
\end{equation}
for $n \geq 1$ and $p \in \llbracket 1, n \rrbracket$. One has the induction formula
\begin{equation}
    \Lucas{n}{p} + \Lucas{n}{p+1} = \Lucas{n+1}{p+1}
\end{equation}
for all $n \geq 1$ and $p \in \llbracket 0, n - 1 \rrbracket$. These coefficients are studied, for instance, in \cite{Robbins2005}. Note that the coefficients of Lucas's triangle should not be confused with Lucas numbers (generally denoted $L_n$ in the literature), which are defined by a variant of the Fibonacci sequence. 

For $k \geq 1$ and $\mu \in H_{(0)}^{4(N-1)}(0,1)$, with $N \geq 1$, we define
\begin{equation}\label{eq:def:a_k_n}
    a_k^n(\mu) = - \frac{1}{2} \sum_{\alpha = 0}^{n - 1} \sum_{\beta = 0}^{n-1-\alpha} (-1)^{n+\beta} \Lucas{n - 1 - \beta}{\alpha} \lambda_k^{n - 1 - \alpha - \beta} \left\langle \mu^{(2\beta)} \mu^{(2(n -1 + \alpha))} , \varphi_k^\prime \right\rangle,
\end{equation}
for all $n \in \llbracket 1, N\rrbracket$. For instance, $a_k^1(\mu) = \frac{1}{2} \left\langle \mu^2 , \varphi_k^\prime \right\rangle$. Our main result is the following theorem. 

\begin{theorem}\label{thm:drift_any_order}
    Let $N \geq 2$, $k \geq 1$, and $\mu \in H_{(0)}^{4N}(0,1)$ such that $\langle \mu, \varphi_k \rangle = 0$. Assume that
    \begin{equation}
        a_k^1(\mu) = \cdots = a_k^{N-1}(\mu) = 0 \quad \text{ and } \quad a_k^N(\mu) \neq 0.
    \end{equation}
    Then there exist $C, T_0, \eta_0 > 0$ such that for all $T \in (0, T_0)$, all $u \in H^{2N-3}(0,T)$ with $\Vert u \Vert_{L^2} \leq 1$, and all $y_0 \in L^2(0,1)$ with $\Vert y_0 \Vert_{L^2} \leq \eta_0$, one has
    \begin{equation}\label{eq:thm:drift_any_order}
        \begin{split}
            & \left\vert \left\langle y(T; y_0, u) - y(T; y_0, 0), \varphi_k \right\rangle - a_k^N(\mu) \Vert u_N \Vert_{L^2}^2 \right\vert \\
            \leq \ & C \Gamma_N(T,u) \Vert u_N \Vert_{L^2}^2 + C \Vert y(T; y_0, u) - y(T; y_0, 0) \Vert_{H^{-1}}^2 + C T^{\frac{1}{4}} \Vert y_0 \Vert_{L^2} \Vert u_1 \Vert_{L^2},
        \end{split}
    \end{equation}
    where $\Gamma_N(T,u) := T + \left\Vert u \right\Vert_{H^{2N-3}} + T^{3-2N} \Vert u \Vert_{L^2}$.
\end{theorem}

The case $N = 1$ is studied in \cite{BurgersFractional} under weaker regularity assumptions on $\mu$, and is therefore omitted here for simplicity. We do not attempt to derive optimal regularity assumptions on $\mu$; since the arguments are rather technical, we instead impose stronger regularity on $\mu$ in order to keep the exposition as clear as possible. For results relating the precise regularity of $\mu$ to obstructions to controllability, we refer to \cite{BurgersFractional}.

The following result establishes the existence of $\mu$ satisfying the assumptions of Theorem \ref{thm:drift_any_order}. 

\begin{proposition}\label{prop:existence_mu}
    For $k \geq 1$ and $N \geq 2$, there exists $\mu \in D(A^\infty)$ such that $\langle \mu, \varphi_k \rangle = 0$ and
    \begin{equation}
        a_k^1(\mu) = \cdots = a_k^{N-1}(\mu) = 0 \quad \text{ and } \quad a_k^N(\mu) \neq 0.
    \end{equation}
\end{proposition}

Finally, the drift estimate \eqref{eq:thm:drift_any_order} implies the following obstruction to controllability.

\begin{corollary}\label{cor:obstruction_small_time}
    Assume that the hypotheses of Theorem \ref{thm:drift_any_order} are satisfied. Then the Burgers equation \eqref{eq:Burgers_intro} is not STLC with controls small in $H^{2N-3}$. More precisely, there exist $c > 0$, $T_1 > 0$, and $\eta_1 > 0$, such that for all $T \in (0, T_1]$, there exists $\delta_1 > 0$ such that for all $u \in H^{2N-3}(0,T)$ satisfying $\Vert u \Vert_{H^{2N-3}(0,T)} \leq \delta_1$, and all $\eta \in (0, \eta_1)$, one has 
    \begin{equation}\label{eq:cor:obstruction_small_time}
        \left\Vert y\!\left(T; \eta \widetilde{y_0}, u \right) \right\Vert_{H^{-1}(0,1)} + \left\Vert y\!\left(T; \eta \widetilde{y_0}, u \right) \right\Vert_{H^{-1}(0,1)}^2 \geq c \left( \eta + \Vert u_N \Vert_{L^2}^2 \right) > 0,
    \end{equation}
    where $\widetilde{y_0} := \sgn \left( a_k^N(\mu) \right) \varphi_k$.
\end{corollary}

\subsection{Proof strategy and comparison with previous results}

\subsubsection{Obstructions to controllability and Lie brackets}

Consider the general finite-dimensional scalar-input control system
\begin{equation}
    \dot{x}(t) = f_0(x(t)) + u(t) f_1(x(t)),
\end{equation}
where $f_0$ and $f_1$ are smooth vector fields on $\mathbb{R}^n$, the control $u$ is scalar-valued, that is, $u(t) \in \mathbb{R}$, and the state satisfies $x(t) \in \mathbb{R}^n$. We assume that $x=0$ is an equilibrium of the uncontrolled system, namely, $f_0(0)=0$, and study the small-time local controllability (STLC) of the system near $0$. When $f_0$ and $f_1$ are analytic, STLC is encoded in the iterated Lie brackets of $f_0$ and $f_1$ evaluated at zero; see \cite{Krener1973}. Several explicit Lie-bracket conditions are known to prevent STLC. For example, the classical work \cite{Sussmann1983} shows that the condition $\left[ f_1, \left[ f_1, f_0 \right] \right](0) \notin S_1$ gives rise to a quadratic obstruction to controllability, where $S_1$ denotes the controllable subspace of the linearized system at $0$. More generally, quadratic obstructions to controllability are classified in \cite{BeauchardMarbach2018,KBFM24Unified}; they are related to the $H^{-n}$ norms of the control, for $1 \leq n \leq \dim S_1$, and to iterated Lie brackets of length $2n+1$.

Quadratic obstructions of integer order have also been identified for several partial differential equations. Examples include bilinear Schrödinger equations with Dirichlet boundary conditions \cite{BeauchardMorancey2014,Bournissou2023_Quad,Coron2006} and with Neumann boundary conditions \cite[Appendix A]{BeauchardMarbachPerrin}, multi-input Schrödinger equations \cite{GherdaouiQuadObstructions26}, a parabolic model \cite[Section 3]{BeauchardMarbach2020}, a fluid mechanics model \cite{CoronKoenigNguyen2024}, and a Korteweg--de Vries system \cite{NiuXiang2025}. In the present article, we show that such quadratic obstructions can occur for the Burgers equation, at any negative integer order. We also prove in Proposition \ref{prop:Lie_brackets} that the assumptions of Theorem \ref{thm:drift_any_order}, and in particular the coefficients $(a_k^n(\mu))$, can be expressed in terms of Lie brackets if $\mu$ is sufficiently regular.

Two other quadratic phenomena have been identified for partial differential equations. First, obstructions may be quantified by fractional Sobolev norms of the control, notably for the Burgers equation in \cite{Marbach2018,Nguyen2025Burgers,BurgersFractional}, for Korteweg--de Vries systems in \cite{CoronKoenigNguyen2022,Nguyen2025}, and for a parabolic system in \cite{BeauchardMarbach2020}. Second, controllability lost at the linearized level may be restored at the quadratic level, as shown in \cite{BeauchardMarbachPerrin} for a bilinear Schrödinger equation with Neumann boundary conditions and in \cite{BeauchardMarbach2020} for a parabolic system. For a broader discussion of local controllability and Lie brackets, we refer to \cite{Beauchard2026ICM}.

\subsubsection{Proof strategy}

The classical approach to deriving these integer-order obstructions consists in performing integrations by parts in the quadratic system $y_2$ so as to produce primitives of the control. This strategy is used, in particular, in \cite[Appendix A]{BeauchardMarbachPerrin}, \cite[Lemma 3.1]{BeauchardMorancey2014}, \cite[Proposition 5.1]{Bournissou2023_Quad}, \cite[Proposition 3.3]{BeauchardMarbach2020}, and \cite[Lemma 3.8]{CoronKoenigNguyen2024}. This argument is used in Proposition \ref{prop:quad_estimates_IPP} below to establish a coercive estimate for the quadratic system. As in \cite{BeauchardMarbachPerrin,Bournissou2023_Quad,BeauchardMarbach2020}, we perform an arbitrary number of integrations by parts. However, the computations leading to the precise definition of the coefficients $(a_k^n(\mu))$ are considerably more involved than those in the cited works; see Lemma \ref{lem:proprietes_H} below. Likewise, establishing the connection between the coefficients $(a_k^n(\mu))$ and iterated Lie brackets requires substantial additional calculations, which are not needed in the cited works. Compare the proof of Proposition \ref{prop:Lie_brackets} with \cite[Lemma A.2]{BeauchardMarbachPerrin}, \cite[Proposition A.2]{Bournissou2023_Quad}, and \cite[Remark 1.7]{BeauchardMarbach2020}.

Once the coercive estimate for the quadratic system has been established, the proof of our main results relies on three types of nonlinear remainder estimates. First, the so-called \emph{quadratic and cubic remainder estimates} quantify the discrepancy between the controlled solution $y(t;0,u)$ and its linearized and quadratic approximations. Second, we estimate the difference between the full solution $y(t;y_0,u)$ and the sum of the free evolution $y(t;y_0,0)$ and the controlled solution with zero initial data $y(t;0,u)$; we refer to these as \emph{decoupling estimates}. Third, \emph{closed-loop estimates} are used to bound some remainder terms in the quadratic coercive estimate by quantities involving the controlled solution. These estimates can only be applied when $\langle \mu, \varphi_j \rangle \neq 0$ for sufficiently many indices $j$. We therefore prove in Lemma \ref{lem:closed_loop_nonzero_moments} that this property always holds under the assumptions of our main result; our approach differs from those used in earlier works, as it adapts the finite-dimensional argument of \cite{KBFM24Unified} to the setting of the Burgers equation.

Finally, for the Schrödinger equation, the existence of a function $\mu$ giving rise to an obstruction quantified by the $H^{-n}$ norm of the control was established in \cite[Proposition A.2]{Bournissou2023_Quad} using oscillatory profiles and Brouwer's fixed-point theorem. Our approach is shorter and more explicit; see Proposition \ref{prop:existence_mu}.

\subsubsection{Controllability of the Burgers equation}

The controlled Burgers equation \eqref{eq:Burgers_intro} with $\mu \equiv 1$ provided the first example of an obstruction to STLC quantified by a fractional norm of the control; see \cite{Marbach2018}. This specific system was recently revisited in \cite{Nguyen2025Burgers}, where the fractional obstruction to STLC was strengthened to an obstruction to \emph{finite-time} controllability in some directions. In the companion paper \cite{BurgersFractional}, the author develops a general framework encompassing these results and establishes new fractional obstruction results.

More broadly, the controllability of the Burgers equation has been studied extensively under various choices of controls, boundary conditions, and notions of controllability. Since the present paper focuses specifically on quadratic obstruction mechanisms for scalar source controls, we do not provide an overview of this broader literature, and we refer to the bibliographical discussions in \cite{Marbach2018,Nguyen2025Burgers}.

\subsection{Organization of the paper}

In Section \ref{sec:properties_quad}, we study the quadratic expansion $y_2$ and prove that it satisfies a coercive estimate. In Section \ref{sec:lie_brackets}, we relate our main result to iterated Lie brackets and use this connection to prove that $\mu$ has sufficiently many non-vanishing Fourier coefficients. In Section  \ref{sec:remainder_estimates}, we derive the nonlinear remainder estimates needed for the proof of our main results. Section \ref{sec:proof_obstructions} is devoted to the proof of the drift estimate and the resulting obstruction to controllability. In Section \ref{sec:proof_existence_mu}, we construct profiles $\mu$ satisfying the assumptions of our main results. Finally, Appendix \ref{sec:appendix_heat_burgers} collects several regularity results for the heat and Burgers equations.
 
\section{Properties of the quadratic expansion}\label{sec:properties_quad}

One has 
\begin{equation}
    \left\langle y_2(T), \varphi_k \right\rangle = \frac{1}{2} \int_0^T e^{-\lambda_k (T-\tau)} \left\langle (y_1(\tau))^2, \varphi_k^\prime \right\rangle \dd \tau,
\end{equation}
and $y_1(\tau) = \int_0^\tau u(t) \Phi (\tau - t) \dd t$, where $\Phi(t) := S(t) \mu$, with $S$ the heat semigroup. In particular, this gives
\begin{equation}\label{eq:forme_quad}
    \left\langle y_2(T), \varphi_k \right\rangle = \int_0^T u(t) \int_0^t u(s) H(t, s) \dd s \dd t,
\end{equation}
with
\begin{equation}
    H(t, s) = \int_t^T e^{-\lambda_k (T-\tau)} \left\langle \Phi (\tau - t) \Phi (\tau - s) , \varphi_k^\prime \right\rangle \dd \tau,
\end{equation}
for $(t,s) \in \Delta_+ := \left\{ (t,s) \in [0, T]^2, s \leq t \right\}$.

\subsection{Properties of the quadratic kernel}

We gather some properties of $H$ in the following lemma.

\begin{lemma}\label{lem:proprietes_H}
    Let $N \geq 1$. If $\mu \in H_{(0)}^{4N}(0,1)$, then for all $(n, m) \in \llbracket 0, N \rrbracket^2$, the derivative $\partial_s^n \partial_t^m H$ is well-defined and belongs to $C^0(\Delta_+; \mathbb{R})$. If $T \in (0,1)$, then
    \begin{equation}\label{eq:lem:propriete_H_2}
        \sum_{0 \leq i, j \leq N} \left\Vert \partial_t^i \partial_s^j H \right\Vert_{L^\infty(\Delta_+)} \leq C,
    \end{equation}
    where $C > 0$ is a constant which depends only on $\mu$, $N$ and $k$. In addition, for all $t \in [0,T]$ and $n \in \llbracket 1, N \rrbracket$, one has
    \begin{equation}\label{eq:lem:formula_H_with_a_k_n}
        \frac{1}{2} \left( \partial_s^{n} \partial_t^{n-1} H - \partial_s^{n-1} \partial_t^{n} H \right) (t,t) = a_k^n(\mu) e^{-\lambda_k(T-t)}.
    \end{equation}
\end{lemma}

\begin{proof}
    \textbf{Step 1:} \textit{Formula for the derivatives of $H$.} Here, we assume that $\mu \in C^\infty_{\mathrm{c}}(0,1)$ to derive formulas for the derivatives of $H$. For simplicity, we write $\Phi^{(n)}(t)$ instead of $\partial_t^n \Phi(t)$, for $t \geq 0$. We prove that for $m \geq 0$, 
    \begin{equation}\label{eq:proof:lem:proprietes_H_step_1_1}
        \begin{split}
            & \partial_t^m H(t,s) = (-1)^{m} \int_t^T e^{-\lambda_k (T-\tau)} \left\langle \Phi^{(m)} (\tau - t) \Phi (\tau - s) , \varphi_k^\prime \right\rangle \dd \tau \\
            & \hspace{1cm} - e^{-\lambda_k (T - t)} \sum_{\alpha = 0}^{m - 1} \sum_{\beta = 0}^{m-1-\alpha} (-1)^\beta \binom{m - 1 - \beta}{\alpha} \lambda_k^{m - 1 - \alpha - \beta} \left\langle \Phi^{(\beta)} (0) \Phi^{(\alpha)} (t - s) , \varphi_k^\prime \right\rangle.
        \end{split}    
    \end{equation}
    This formula can be obtained directly by differentiation, but this requires calculating multiple sums involving binomial coefficients. This argument is elementary but quite involved, so for clarity, we give a proof of \eqref{eq:proof:lem:proprietes_H_step_1_1} by induction. The formula is clearly true for $m = 0$ (and also for $m=1$). Assume that \eqref{eq:proof:lem:proprietes_H_step_1_1} holds true for some $m \geq 0$. Then
    \begin{align}
        & \partial_t^{m+1} H(t,s) = (-1)^{m+1} \int_t^T e^{-\lambda_k (T-\tau)} \left\langle \Phi^{(m+1)} (\tau - t) \Phi (\tau - s) , \varphi_k^\prime \right\rangle \dd \tau \\
        & \hspace{1cm} - (-1)^m e^{-\lambda_k (T - t)} \left\langle \Phi^{(m)} (0) \Phi (t - s) , \varphi_k^\prime \right\rangle \\
        & \hspace{1cm} - \lambda_k e^{-\lambda_k (T - t)} \sum_{\alpha = 0}^{m - 1} \sum_{\beta = 0}^{m-1-\alpha} (-1)^\beta \binom{m - 1 - \beta}{\alpha} \lambda_k^{m - 1 - \alpha - \beta} \left\langle \Phi^{(\beta)} (0) \Phi^{(\alpha)} (t - s) , \varphi_k^\prime \right\rangle \\
        & \hspace{1cm} - e^{-\lambda_k (T - t)} \sum_{\alpha = 0}^{m - 1} \sum_{\beta = 0}^{m-1-\alpha} (-1)^\beta \binom{m - 1 - \beta}{\alpha} \lambda_k^{m - 1 - \alpha - \beta} \left\langle \Phi^{(\beta)} (0) \Phi^{(\alpha + 1)} (t - s) , \varphi_k^\prime \right\rangle \label{eq:proof:lem:proprietes_H_step_1_2}
    \end{align}
    To improve readability, we move the integral term to the left-hand side. The change of index $\alpha^\prime = \alpha + 1$ in the sum appearing in line \eqref{eq:proof:lem:proprietes_H_step_1_2} gives
    \begin{equation}
        \begin{split}
            & \partial_t^{m+1} H(t,s) - (-1)^{m+1} \int_t^T e^{-\lambda_k (T-\tau)} \left\langle \Phi^{(m+1)} (\tau - t) \Phi (\tau - s) , \varphi_k^\prime \right\rangle \dd \tau \\
            = \ & - (-1)^m e^{-\lambda_k (T - t)} \left\langle \Phi^{(m)} (0) \Phi (t - s) , \varphi_k^\prime \right\rangle \\
            & - e^{-\lambda_k (T - t)} \sum_{\alpha = 0}^{m - 1} \sum_{\beta = 0}^{m-1-\alpha} (-1)^\beta \binom{m - 1 - \beta}{\alpha} \lambda_k^{m - \alpha - \beta} \left\langle \Phi^{(\beta)} (0) \Phi^{(\alpha)} (t - s) , \varphi_k^\prime \right\rangle \\
            & - e^{-\lambda_k (T - t)} \sum_{\alpha = 1}^{m} \sum_{\beta = 0}^{m-\alpha} (-1)^\beta \binom{m - 1 - \beta}{\alpha - 1} \lambda_k^{m - \alpha - \beta} \left\langle \Phi^{(\beta)} (0) \Phi^{(\alpha)} (t - s) , \varphi_k^\prime \right\rangle 
        \end{split}
    \end{equation}
    We gather the sums with index $\alpha \in \llbracket 1, m-1 \rrbracket$ and $\beta \in \llbracket 0, m - 1 - \alpha \rrbracket$, by using the identities
    \begin{equation}
        \sum_{\alpha = 0}^{m - 1} \sum_{\beta = 0}^{m-1-\alpha} f(\alpha, \beta) = \sum_{\beta = 0}^{m - 1} f(0, \beta) + \sum_{\alpha = 1}^{m - 1} \sum_{\beta = 0}^{m - 1 -\alpha} f(\alpha, \beta),
    \end{equation}
    and
    \begin{equation}
        \sum_{\alpha = 1}^{m} \sum_{\beta = 0}^{m-\alpha} f(\alpha, \beta) = \sum_{\alpha = 1}^{m} f(\alpha, m - \alpha) + \sum_{\alpha = 1}^{m - 1} \sum_{\beta = 0}^{m - 1 -\alpha} f(\alpha, \beta),
    \end{equation}
    which are valid for any function $f : \mathbb{N}^2 \rightarrow \mathbb{R}$, together with the formula $\binom{m - 1 - \beta}{\alpha - 1}  + \binom{m - 1 - \beta}{\alpha} = \binom{m - \beta}{\alpha}$, to find
    \begin{align}
        & \partial_t^{m+1} H(t,s) - (-1)^{m+1} \int_t^T e^{-\lambda_k (T-\tau)} \left\langle \Phi^{(m+1)} (\tau - t) \Phi (\tau - s) , \varphi_k^\prime \right\rangle \dd \tau \\
        = \ & - (-1)^m e^{-\lambda_k (T - t)} \left\langle \Phi^{(m)} (0) \Phi (t - s) , \varphi_k^\prime \right\rangle \label{eq:proof:lem:proprietes_H_step_1_4}\\
        & - e^{-\lambda_k (T - t)} \sum_{\beta = 0}^{m-1} (-1)^\beta \lambda_k^{m - \beta} \left\langle \Phi^{(\beta)} (0) \Phi (t - s) , \varphi_k^\prime \right\rangle \label{eq:proof:lem:proprietes_H_step_1_5}\\
        & - e^{-\lambda_k (T - t)} \sum_{\alpha = 1}^{m} (-1)^{m-\alpha} \left\langle \Phi^{(m - \alpha)}(0) \Phi^{(\alpha)} (t - s) , \varphi_k^\prime \right\rangle \\
        & - e^{-\lambda_k (T - t)} \sum_{\alpha = 1}^{m - 1} \sum_{\beta = 0}^{m - 1 -\alpha} (-1)^\beta \binom{m - \beta}{\alpha} \lambda_k^{m - \alpha - \beta} \left\langle \Phi^{(\beta)} (0) \Phi^{(\alpha)} (t - s) , \varphi_k^\prime \right\rangle .
    \end{align}
    Gathering \eqref{eq:proof:lem:proprietes_H_step_1_4} and \eqref{eq:proof:lem:proprietes_H_step_1_5}, and using the identity
    \begin{equation}
        \sum_{\beta = 0}^{m} f(0, \beta) + \sum_{\alpha = 1}^{m} f(\alpha, m - \alpha) + \sum_{\alpha = 1}^{m - 1} \sum_{\beta = 0}^{m - 1 -\alpha} f(\alpha, \beta) = \sum_{\alpha = 0}^{m} \sum_{\beta = 0}^{m-\alpha} f(\alpha, \beta),
    \end{equation}
    one obtains \eqref{eq:proof:lem:proprietes_H_step_1_1} for $m + 1$. Note that \eqref{eq:proof:lem:proprietes_H_step_1_1} implies that for $m \geq 0$ and $n \geq 0$,
    \begin{equation}\label{eq:proof:lem:proprietes_H_step_1_3}
        \begin{split}
            & \partial_s^n \partial_t^m H(t,s) = (-1)^{m + n} \int_t^T e^{-\lambda_k (T-\tau)} \left\langle \Phi^{(m)} (\tau - t) \Phi^{(n)} (\tau - s) , \varphi_k^\prime \right\rangle \dd \tau \\
            & \hspace{0.5cm} - e^{-\lambda_k (T - t)} \sum_{\alpha = 0}^{m - 1} \sum_{\beta = 0}^{m-1-\alpha} (-1)^{n+\beta} \binom{m - 1 - \beta}{\alpha} \lambda_k^{m - 1 - \alpha - \beta} \left\langle \Phi^{(\beta)} (0) \Phi^{(n + \alpha)} (t - s) , \varphi_k^\prime \right\rangle.
        \end{split}    
    \end{equation}

    \textbf{Step 2:} \textit{Regularity of $H$.}
    Now, consider $\mu \in H_{(0)}^{4N}(0,1)$. One has
    \begin{equation}
        \Phi \in \bigcap_{n = 0}^{2N} C^n([0, T]; H_{(0)}^{4N - 2n}(0,1)),
    \end{equation}
    and $\partial_t^n \Phi(0) = \partial_x^{2n} \Phi(0) = \mu^{(2n)}$, for all $n \in \llbracket 0, 2N \rrbracket$. In particular, for all $(n, m) \in \llbracket 0, 2N \rrbracket^2$, the functions
    \begin{equation}
        (t_1, t_2) \longmapsto \left\langle \partial_t^n \Phi (t_1) \partial_t^m \Phi (t_2) , \varphi_k^\prime \right\rangle \quad \text{ and } \quad 
        t_1 \longmapsto \left\langle \partial_t^n \Phi (0) \partial_t^m \Phi (t_1) , \varphi_k^\prime \right\rangle
    \end{equation}
    are continuous on $[0, T]^2$ and $[0,T]$ respectively. This (largely) implies that all the computations made to obtain \eqref{eq:proof:lem:proprietes_H_step_1_3} are valid if $n,m \leq N$, and this proves the regularity claimed for $H$. Furthermore, Lemma \ref{lem:well-posedness_heat_L1L2} (with $f = 0$) gives
    \begin{equation}
        \left\Vert \partial_t^n \Phi \right\Vert_{L^\infty((0,T);L^2)} \leq C \left\Vert \mu^{(2n)} \right\Vert_{L^2},
    \end{equation}
    for all $n \in \llbracket 0, 2 N \rrbracket$, where $C > 0$ is an absolute constant, and if $T \in (0,1)$, then
    \begin{equation}
        \begin{split}
            & \sum_{0 \leq n, m \leq N} \left\vert \partial_s^n \partial_t^m H(t,s) \right\vert \\
            \leq \ & C \left\Vert \mu \right\Vert_{H_{(0)}^{2N-2}} \sum_{n = 0}^{2N-1} \left\Vert \partial_t^{n} \Phi \right\Vert_{L^\infty((0,T);L^2)} 
             + \sum_{0 \leq n, m \leq N}  \int_0^T \left\Vert \partial_t^n \Phi (\tau - t) \right\Vert_{L^2} \left\Vert \partial_t^m \Phi(\tau - s) \right\Vert_{L^2} \dd \tau \\
            \leq \ & C \left\Vert \mu \right\Vert_{H_{(0)}^{4N-2}}^2 ,
        \end{split}
    \end{equation}
    where $C > 0$ is a constant which depends only on $k$ and $N$. This gives \eqref{eq:lem:propriete_H_2}. 
    
    \textbf{Step 3:} \textit{Computation of $a_k^n(\mu)$.}
    Finally, we compute $ \left( \partial_s^{n} \partial_t^{n-1} - \partial_s^{n-1} \partial_t^{n} \right) H (t,t)$. One has
    \begin{equation}
        \begin{split}
            & \left( \partial_s^{n} \partial_t^{n-1} - \partial_s^{n-1} \partial_t^{n} \right) H (t,t) \\
            =  & - e^{-\lambda_k (T - t)} \sum_{\alpha = 0}^{n - 2} \sum_{\beta = 0}^{n-2-\alpha} (-1)^{n+\beta} \binom{n - 2 - \beta}{\alpha} \lambda_k^{n - 2 - \alpha - \beta} \left\langle \Phi^{(\beta)} (0) \Phi^{(n + \alpha)} (0) , \varphi_k^\prime \right\rangle \\
            & + e^{-\lambda_k (T - t)} \sum_{\alpha = 0}^{n - 1} \sum_{\beta = 0}^{n-1-\alpha} (-1)^{n-1+\beta} \binom{n - 1 - \beta}{\alpha} \lambda_k^{n - 1 - \alpha - \beta} \left\langle \Phi^{(\beta)} (0) \Phi^{(n -1 + \alpha)} (0) , \varphi_k^\prime \right\rangle.
        \end{split}
    \end{equation}
    The change of index $\alpha^\prime = \alpha + 1$ in the first sum gives
    \begin{equation}
        \begin{split}
            & \left( \partial_s^{n} \partial_t^{n-1} - \partial_s^{n-1} \partial_t^{n} \right) H (t,t) \\
            =  & - e^{-\lambda_k (T - t)} \sum_{\alpha = 1}^{n - 1} \sum_{\beta = 0}^{n-1-\alpha} (-1)^{n+\beta} \binom{n - 2 - \beta}{\alpha - 1} \lambda_k^{n - 1 - \alpha - \beta} \left\langle \Phi^{(\beta)} (0) \Phi^{(n -1 + \alpha)} (0) , \varphi_k^\prime \right\rangle \\
            & + e^{-\lambda_k (T - t)} \sum_{\alpha = 0}^{n - 1} \sum_{\beta = 0}^{n-1-\alpha} (-1)^{n-1+\beta} \binom{n - 1 - \beta}{\alpha} \lambda_k^{n - 1 - \alpha - \beta} \left\langle \Phi^{(\beta)} (0) \Phi^{(n -1 + \alpha)} (0) , \varphi_k^\prime \right\rangle.
        \end{split}
    \end{equation}
    and gathering the sum with index $\alpha \in \llbracket 1, n-1 \rrbracket$, one finds
    \begin{equation}
        \begin{split}
            & \left( \partial_s^{n} \partial_t^{n-1} - \partial_s^{n-1} \partial_t^{n} \right) H (t,t) \\
            =  & - e^{-\lambda_k (T - t)} \sum_{\alpha = 1}^{n - 1} \sum_{\beta = 0}^{n-1-\alpha} (-1)^{n+\beta} \Lucas{n - 1 - \beta}{\alpha} \lambda_k^{n - 1 - \alpha - \beta} \left\langle \Phi^{(\beta)} (0) \Phi^{(n -1 + \alpha)} (0) , \varphi_k^\prime \right\rangle \\
            & + e^{-\lambda_k (T - t)} \sum_{\beta = 0}^{n-1} (-1)^{n-1+\beta} \lambda_k^{n - 1 - \beta} \left\langle \Phi^{(\beta)} (0) \Phi^{(n -1)} (0) , \varphi_k^\prime \right\rangle\\
            =  & - e^{-\lambda_k (T - t)} \sum_{\alpha = 0}^{n - 1} \sum_{\beta = 0}^{n-1-\alpha} (-1)^{n+\beta} \Lucas{n - 1 - \beta}{\alpha} \lambda_k^{n - 1 - \alpha - \beta} \left\langle \Phi^{(\beta)} (0) \Phi^{(n -1 + \alpha)} (0) , \varphi_k^\prime \right\rangle.
        \end{split}
    \end{equation}
    Using $\Phi^{(p)}(0) = \partial_t^p \Phi(0) = \partial_x^{2p} \Phi(0) = \mu^{(2p)}$, one finds the claimed expression of $a_k^n(\mu)$.
\end{proof}

\subsection{Coercive estimate}

We prove the following coercive estimate. The proof relies on repeated integration by parts, following the approach used in \cite[Proposition 3.3]{BeauchardMarbach2020}, \cite[Proposition 5.1]{Bournissou2023_Quad}, and \cite[Proposition A.3]{BeauchardMarbachPerrinSchrod}.

\begin{proposition}\label{prop:quad_estimates_IPP}
    Let $N \geq 1$, $k \geq 1$, and $\mu \in H_{(0)}^{4N}(0,1)$ such that $N$ is the minimal integer such that $a_k^N(\mu) \neq 0$. Then there exists $C > 0$ such that for all $T \in (0, 1)$ and all $u \in L^2(0, T)$,
    \begin{equation}
        \left\vert \left\langle y_2(T), \varphi_k \right\rangle - a_k^N(\mu) \Vert u_N \Vert_{L^2}^2 \right\vert \leq C \left( T \Vert u_N \Vert_{L^2}^2 + \left\vert u_1(T) \right\vert^2 + \cdots + \left\vert u_N(T) \right\vert^2 \right).
    \end{equation}
\end{proposition}

\begin{proof}
    In this proof, the symbol $\lesssim$ is used for constants that depend only on $\mu$, $N$ and $k$.
    The proof is based on a recursive application of the integration by parts formula
    \begin{equation}\label{eq:proof:prop_quad_expansion}
        \begin{split}
            \int_0^T u(t) \int_0^t u(s) H(t, s) \dd s \dd t = \ & \frac{1}{2} \int_0^T u_1(t)^2 \left( \partial_s H(t,t) - \partial_t H(t,t) \right) \dd t \\
            & + \int_0^T u_1(t) \int_0^t u_1(s) \partial_t \partial_s H(t, s) \dd s \dd t \\
            & + \frac{1}{2} u_1(T)^2 H(T,T) - u_1(T) \int_0^T u_1(s) \partial_s H(T,s) \dd s
        \end{split}
    \end{equation}
    in \eqref{eq:forme_quad}. Set 
    \begin{equation}
        \gamma_{i,j}^{n, T} = \int_0^T u_n(s) \partial_t^i \partial_s^j H(T,s) \dd s.
    \end{equation}
    Applying \eqref{eq:proof:prop_quad_expansion} with $(u, H)$ replaced by $(u_n, \partial_t^n \partial_s^n H)$ yields the term $u_{n+1}(T) \gamma_{n, n+1}^{n+1, T}$. Furthermore, for $i \in \llbracket 0, N \rrbracket$ and $j, n \in \llbracket 0, N - 1 \rrbracket$, integration by parts gives
    \begin{equation}
        \gamma_{i,j}^{n, T} = - \gamma_{i, j + 1}^{n + 1, T} + u_{n+1}(T) \partial_t^i \partial_s^{j} H(T,T).
    \end{equation}
    Summarizing, one obtains
    \begin{equation}\label{eq:IPP_iterated_kernel_H}
        \begin{split}
            \int_0^T u(t) \int_0^t u(s) H(t, s) \dd s \dd t = \ & \frac{1}{2} \sum_{n = 1}^N \int_0^T u_n(t)^2 \left( \partial_s^{n} \partial_t^{n-1} - \partial_s^{n-1} \partial_t^{n} \right) H (t,t) \dd t \\
            & + \int_0^T u_N(t) \int_0^t u_N(s) \partial_t^N \partial_s^N H(t, s) \dd s \dd t \\
            & + Q_{N, T} \left(u_1(T), \cdots u_N(T), \gamma_{0, N}^{N, T}, \cdots \gamma_{N-1, N}^{N, T} \right),
        \end{split}
    \end{equation}
    where $Q_{N, T}$ denotes a quadratic form on $\mathbb{R}^{2N}$ whose coefficients are uniformly bounded for $T \in (0, 1)$.
    Since $a_k^1 = \cdots = a_k^{N-1} = 0$, Lemma \ref{lem:proprietes_H} gives
    \begin{equation}
        \frac{1}{2} \sum_{n = 1}^N \int_0^T u_n(t)^2 \left( \partial_s^{n-1} \partial_t^{n-1} \left( \partial_s H - \partial_t H \right) \right)(t,t) \dd t = a_k^N(\mu) \int_0^T u_N(t)^2 e^{-\lambda_k(T-t)} \dd t ,
    \end{equation}
    implying
    \begin{equation}
        \left\vert \frac{1}{2} \sum_{n = 1}^N \int_0^T u_n(t)^2 \left( \partial_s^{n-1} \partial_t^{n-1} \left( \partial_s H - \partial_t H \right) \right)(t,t) \dd t - a_k^N(\mu) \Vert u_N \Vert_{L^2}^2 \right\vert \lesssim T \Vert u_N \Vert_{L^2}^2.
    \end{equation}
    Using the Cauchy-Schwarz inequality and the estimates on the derivatives of $H$ given by Lemma \ref{lem:proprietes_H}, one finds 
    \begin{equation}
        \begin{split}
            & \left\vert \int_0^T u_N(t) \int_0^t u_N(s) \partial_t^N \partial_s^N H(t, s) \dd s \dd t \right\vert + \left\vert Q_{N, T} \left(u_1(T), \cdots u_N(T), \gamma_{0, N}^{N, T}, \cdots \gamma_{N-1, N}^{N, T} \right) \right\vert \\
            \lesssim \ & T \Vert u_N \Vert_{L^2}^2 + \left\vert u_1(T) \right\vert^2 + \cdots + \left\vert u_N(T) \right\vert^2.
        \end{split}
    \end{equation}
    This completes the proof of Proposition \ref{prop:quad_estimates_IPP}.
\end{proof}

\section{Lie brackets, non-vanishing Fourier coefficients}\label{sec:lie_brackets}

\subsection{Interpretation in terms of Lie brackets}

\begin{definition}[Lie brackets]
    Let $E$ be a Fréchet space, and let $f,g \in C^1(E, E)$. Then the Lie bracket of $f$ and $g$ is defined by
    \begin{equation}
        \begin{array}{cccc}
            [f, g] : & E & \longrightarrow & E \\
            & y & \longmapsto & \D g(y) \left( f(y) \right) - \D f(y) \left( g(y) \right)
        \end{array},
    \end{equation}
    where $\D f(y)$ and $\D g(y)$ denote the differentials of $f$ and $g$ at $y$. For $n \geq 0$, we define $\ad_f^n(g)$ recursively as follows: we set $\ad_f^0(g) = g$, and for $n \geq 0$, if $\ad_f^n(g)$ is well-defined and belongs to $C^1(E, E)$, then we set
    \begin{equation}
        \ad_f^{n+1}(g) = [f, \ad_f^n(g)].
    \end{equation}
\end{definition}

\begin{proposition}[Interpretation of $a_k^N(\mu)$ in terms of Lie brackets]\label{prop:Lie_brackets}
    Let $\mu \in D(A^\infty)$ and $k \geq 1$. For $y \in D(A^\infty)$, set $f_0(y) = y^{\prime \prime} - y y^\prime$ and $f_1(y) = \mu$. For $n \geq 0$, set $M_n = \ad_{f_0}^{n}(f_1)$, and for $N \geq 1$, set $W_N = \left[M_{N-1}, M_N\right]$. Then, for all $N \geq 1$, one has
    \begin{equation}
        a_k^N(\mu) = - \frac{1}{2} \left\langle W_N(0) , \varphi_k \right\rangle.
    \end{equation}
\end{proposition}

\begin{remark}
    Proposition \ref{prop:Lie_brackets} is the main ingredient in the proof of Lemma \ref{lem:closed_loop_nonzero_moments} below. It will only be applied to $\mu$ of the form $\sum_{j \in J} \langle \mu, \varphi_j \rangle \varphi_j$, where $J$ is a finite subset of $\mathbb{N}^\ast$. In this setting, Proposition \ref{prop:Lie_brackets} provides the Lie-bracket interpretation of the coefficient $a_k^N(\mu)$ introduced in Lemma \ref{lem:proprietes_H}. This identity is the infinite-dimensional counterpart of the standard finite-dimensional relation between an $H^{-N}$-quadratic drift and Lie brackets. For completeness, we nevertheless provide a direct proof.
\end{remark}

\begin{proof}  
    One has $\D f_1(y) (Y) = 0$, 
    \begin{equation}
        \D f_0(y) (Y) = Y^{\prime \prime} - y Y^\prime - y^\prime Y, \quad \text{ and } \quad \D^2 f_0(0) (Y, Z) = - Y Z^\prime - Y^\prime Z = - (YZ)^{\prime},
    \end{equation}
    and by definition,
    \begin{equation}\label{eq:proof:prop:Lie_brackets_1}
        W_N(0) = \D M_{N}(0)( M_{N-1}(0) ) - \D M_{N-1}(0)( M_{N}(0) ).
    \end{equation}

    \textbf{Step 1.} \textit{Induction formulas.} 
    First, by definition, $M_0(0) = f_1(0) = \mu$, and for $n \geq 0$,
    \begin{equation}
        M_{n+1}(0) = \left[f_0, M_n \right](0) =  \D M_n(0) \left( f_0(0) \right) - \D f_0(0) \left( M_n(0) \right) = - \left( M_n(0) \right)^{\prime \prime},
    \end{equation}
    yielding
    \begin{equation}\label{eq:proof:prop:Lie_brackets_2}
        M_n(0) = (-1)^n \mu^{(2n)}.
    \end{equation}

    Second, for $n \geq 0$,
    \begin{equation}
        M_{n+1}(y) = \left[f_0, M_n \right](y) =  \D M_n(y) \left( f_0(y) \right) - \D f_0(y) \left( M_n(y) \right),
    \end{equation}
    yielding
    \begin{align}
        \D M_{n+1}(0)(Y) = \ & \D^2 M_n(0) \left( f_0(0), Y \right) + \D M_n(0) \left( \D f_0(0)(Y) \right) \\
        & - \D^2 f_0(0) \left( M_n(0), Y \right) - \D f_0(0) \left( \D M_n(0)(Y) \right) \\
        = \ & \D M_n(0) \left( Y^{\prime \prime} \right) + (-1)^n \left( \mu^{(2n)} Y \right)^{\prime} - \left( \D M_n(0)(Y) \right)^{\prime \prime}. \label{eq:proof:prop:Lie_brackets_3}
    \end{align}
    Note that those formulas hold for $y, Y \in D(A^\infty)$, and that $\D M_n(0)(Y) \in D(A^\infty)$ for all $Y \in D(A^\infty)$. This property is automatic since $M_n: D(A^\infty) \rightarrow D(A^\infty)$; alternatively, it can be verified directly from the fact that $(YZ)^\prime \in D(A^\infty)$ for all $Y,Z \in D(A^\infty)$. Using \eqref{eq:proof:prop:Lie_brackets_1}, \eqref{eq:proof:prop:Lie_brackets_2} and \eqref{eq:proof:prop:Lie_brackets_3}, one can compute recursively $W_N(0)$ for all $N \geq 1$. Note that we do not need an explicit formula for $W_N(0)$ for $N \geq 1$; we only need to compute $\left\langle W_N(0) , \varphi_k \right\rangle$, which allows us to replace second derivatives by multiplication by $-\lambda_k$. 
    
    \textbf{Step 2.} \textit{Computation of $\left\langle \D M_n(0) (Y), \varphi_k \right\rangle$ for $n \geq 0$ and $Y \in D(A^\infty)$.} 
    To lighten the notations, set $F_n(Y) = \left\langle \D M_n(0) (Y), \varphi_k \right\rangle$. Since $M_0(y) = f_1(y) = \mu$, one has $F_0(Y) = 0$. Using \eqref{eq:proof:prop:Lie_brackets_3} together with integration by parts, one obtains
    \begin{equation}
        F_{n+1}(Y) = F_n \left( Y^{\prime \prime} \right) + (-1)^n \left\langle \left( \mu^{(2n)} Y \right)^{\prime}, \varphi_k \right\rangle + \lambda_k F_n(Y).
    \end{equation}
    Note that no boundary term appears when performing the integration by parts, since $\D M_n(0) (Y) \in H_0^1(0,1)$ for all $n \geq 0$.
    Setting $G_n(Y) = (-1)^n \left\langle \left( \mu^{(2n)} Y \right)^{\prime}, \varphi_k \right\rangle$, this relation can be rewritten as
    \begin{equation}
        F_{n+1}(Y) = F_n \left(  Y^{\prime \prime} + \lambda_k Y \right) + G_n(Y).
    \end{equation}
    A straightforward induction gives $F_n = \sum_{p = 0}^{n-1} G_p \circ \left( \partial_x^2 + \lambda_k \right)^{n-1-p}$, implying
    \begin{align}
        F_n(Y) & = \sum_{p = 0}^{n-1} \sum_{q = 0}^{n-1-p} (-1)^p \binom{n-1-p}{q} \lambda_k^q \left\langle \left( \mu^{(2p)} Y^{(2(n-1-p-q))} \right)^{\prime}, \varphi_k \right\rangle \\
         & = - \sum_{p = 0}^{n-1} \sum_{q = 0}^{n-1-p} (-1)^p \binom{n-1-p}{q} \lambda_k^q \left\langle \mu^{(2p)} Y^{(2(n-1-p-q))} , \varphi_k^\prime \right\rangle . \label{eq:proof:prop:Lie_brackets_4} 
    \end{align}

    \textbf{Step 3.} \textit{End of the proof.}  
    Using \eqref{eq:proof:prop:Lie_brackets_1}, \eqref{eq:proof:prop:Lie_brackets_2} and \eqref{eq:proof:prop:Lie_brackets_4}, one finds 
    \begin{equation}
        \begin{split}
            \left\langle W_N(0) , \varphi_k \right\rangle & = (-1)^{N+1} F_N \left( \mu^{(2(N-1))} \right) + (-1)^{N+1} F_{N-1} \left( \mu^{(2N)} \right)  \\
            & = (-1)^{N} \sum_{p = 0}^{N-1} \sum_{q = 0}^{N-1-p} (-1)^p \binom{N-1-p}{q} \lambda_k^q \left\langle \mu^{(2p)} \mu^{(2(2N-2-p-q))} , \varphi_k^\prime \right\rangle \\
            & \hspace{0.5cm} +  (-1)^{N} \sum_{p = 0}^{N-2} \sum_{q = 0}^{N-2-p} (-1)^p \binom{N-2-p}{q} \lambda_k^q \left\langle \mu^{(2p)} \mu^{(2(2N-2-p-q))} , \varphi_k^\prime \right\rangle\\
            & = (-1)^{N} \sum_{p = 0}^{N-1} \sum_{q = 0}^{N-1-p} c_{p,q}^N \lambda_k^q \left\langle \mu^{(2p)} \mu^{(2(2N-2-p-q))} , \varphi_k^\prime \right\rangle,
        \end{split}
    \end{equation}  
    with $c_{p,q}^N = (-1)^p \binom{N-1-p}{q} + (-1)^p \binom{N-2-p}{q} \mathbbm{1}_{p \leq N-2} \mathbbm{1}_{q \leq N-2-p}$. Making the change of index $p = \beta$, and then $q = N - 1 - \alpha - \beta$, one finds
    \begin{equation}
        \left\langle W_N(0) , \varphi_k \right\rangle = (-1)^{N} \sum_{\alpha = 0}^{N-1} \sum_{\beta = 0}^{N - 1 - \alpha}  c_{\beta, N - 1 - \alpha - \beta}^N \lambda_k^{N - 1 - \alpha - \beta} \left\langle \mu^{(2\beta)} \mu^{(2(N-1 + \alpha))} , \varphi_k^\prime \right\rangle.
    \end{equation}
    Finally, for $0 \leq \alpha, \beta \leq N - 1$ with $\alpha + \beta \leq N-1$, one has
    \begin{equation}
        \begin{split}
            c_{\beta, N - 1 - \alpha - \beta}^N & = (-1)^\beta \binom{N-1-\beta}{N - 1 - \alpha - \beta} + (-1)^\beta \binom{N-2-\beta}{N - 1 - \alpha - \beta} \mathbbm{1}_{\beta \leq N-2} \mathbbm{1}_{\alpha \geq 1} \\
            & = (-1)^\beta \binom{N-1-\beta}{\alpha} + (-1)^\beta \binom{N-2-\beta}{\alpha - 1} \mathbbm{1}_{\beta \leq N-2} \mathbbm{1}_{\alpha \geq 1} \\
            & = (-1)^\beta \Lucas{N-1-\beta}{\alpha},
        \end{split}
    \end{equation}
    and this completes the proof.
\end{proof}

\subsection{Non-vanishing Fourier coefficients}

The coercive quadratic estimate (Proposition \ref{prop:quad_estimates_IPP}) involves the remainder terms $u_1(T)^2$, ..., $u_N(T)^2$, which will be estimated using closed-loop estimates (Lemma \ref{lem:closed_loop}). We prove the following technical result, which will be essential for verifying the assumptions of that lemma. 

\begin{lemma}\label{lem:closed_loop_nonzero_moments}
    Let $N \geq 1$, $k \geq 1$, and $\mu \in H_{(0)}^{4N}(0,1)$ be such that $a_k^1(\mu) = \cdots = a_k^{N-1}(\mu) = 0$ and $a_k^N(\mu) \neq 0$. Then there exists $J \subset \mathbb{N}^\ast$, with $\vert J \vert \geq N$, such that $\langle \mu, \varphi_j \rangle \neq 0$ for all $j \in J$.
\end{lemma}

\begin{proof}
    For simplicity, we write $a_k^n := a_k^n(\mu)$ in this proof. Assume by contradiction that there exists $J \subset \mathbb{N}^\ast$, with $\vert J \vert \leq N-1$, such that 
    \begin{equation}
        \mu = \sum_{j \in J} \alpha_j \varphi_j,
    \end{equation}
    for some real coefficients $(\alpha_j)_{j \in J}$. In particular, $\mu \in D(A^\infty)$ and we know by Proposition \ref{prop:Lie_brackets} that $a_k^n$ can be expressed in terms of Lie brackets. We will show that the condition $a_k^1 = \cdots = a_k^{N-1} = 0$ implies $a_k^N = 0$, which gives a contradiction. The proof involves only a finite number of vectors from the basis $(\varphi_j)$, and we adapt the proof of the finite dimensional result \cite[Lemma 6.4]{KBFM24Unified}.
    
    Recall that for $y \in D(A^\infty)$, $n \geq 0$, and $N \geq 1$, we write $f_0(y) = y^{\prime \prime} - y y^\prime$, $f_1(y) = \mu$, $M_n = \ad_{f_0}^{n}(f_1)$, and $W_N = \left[M_{N-1}, M_N\right]$. Proposition \ref{prop:Lie_brackets} gives
    \begin{equation}\label{eq:proof:lem:closed_loop_nonzero_moments_1}
        a_k^N = - \frac{1}{2} \left\langle W_N(0) , \varphi_k \right\rangle.
    \end{equation}
    
    \textbf{Step 1.} \textit{We prove that $\left\langle \left[ M_\alpha, M_\beta \right](0), \varphi_k \right\rangle = 0$ for all $\alpha, \beta \geq 0 $ such that $\alpha + \beta \leq 2N-2$.}
    Equivalently, we prove by double induction on $n \in \llbracket 0, 2N-2 \rrbracket$ that
    \begin{equation}\label{eq:proof:lem:closed_loop_nonzero_moments_2}
        \left\langle \left[ M_\alpha, M_{\alpha + n} \right](0), \varphi_k \right\rangle = 0, \quad \text{for all $\alpha \geq 0$ such that $2 \alpha \leq 2 N - 2 - n$}.
    \end{equation}
    It holds true for $n = 0$. If $\alpha \geq 0$ satisfies $2 \alpha \leq 2 N - 3$, then 
    \begin{equation}
        \left\langle \left[ M_\alpha, M_{\alpha + 1} \right](0), \varphi_k \right\rangle = \left\langle W_{\alpha + 1}(0), \varphi_k \right\rangle = -2 a_k^{\alpha + 1} = 0,
    \end{equation}
    since $\alpha + 1 \leq N-1$, implying that \eqref{eq:proof:lem:closed_loop_nonzero_moments_2} also holds true for $n=1$. 
    
    Now, let $n \in \llbracket 1, 2N-3 \rrbracket$  be such that \eqref{eq:proof:lem:closed_loop_nonzero_moments_2} holds true for $n$ and $n-1$. Let $\alpha \geq 0$ be such that $2 \alpha \leq 2 N - 2 - (n+1)$. The Jacobi identity gives
    \begin{align}
        \left\langle \left[ M_\alpha, M_{\alpha + n + 1} \right](0), \varphi_k \right\rangle 
        & = \left\langle \left[ M_\alpha, [f_0, M_{\alpha + n}] \right](0), \varphi_k \right\rangle \\
        & = \left\langle \left[ [M_\alpha, f_0], M_{\alpha + n} \right](0), \varphi_k \right\rangle 
          + \left\langle \left[ f_0, [M_\alpha, M_{\alpha + n}] \right](0), \varphi_k \right\rangle \\
        & = - \left\langle \left[ M_{\alpha+1}, M_{\alpha + n} \right](0), \varphi_k \right\rangle 
          - \left\langle \D f_0(0) \left( [M_\alpha, M_{\alpha + n}](0) \right), \varphi_k \right\rangle, \label{eq:proof:lem:closed_loop_nonzero_moments_3}
    \end{align}
    since $f_0(0) = 0$.
    The first term of \eqref{eq:proof:lem:closed_loop_nonzero_moments_3} vanishes since we assume that \eqref{eq:proof:lem:closed_loop_nonzero_moments_2} holds true for $n-1$, and since $2(\alpha + 1) \leq 2 N - 2 - (n-1)$. Integration by parts in the second term gives
    \begin{equation}
        - \left\langle \D f_0(0) \left( [M_\alpha, M_{\alpha + n}](0) \right), \varphi_k \right\rangle = \lambda_k \left\langle [M_\alpha, M_{\alpha + n}](0), \varphi_k \right\rangle.
    \end{equation}
    Note that no boundary term appears since $[M_\alpha, M_{\alpha + n}](0) \in D(A^\infty)$. As we assume that \eqref{eq:proof:lem:closed_loop_nonzero_moments_2} holds true for $n$, and since $2 \alpha \leq 2 N - 2 - n$, this second term also vanishes. This proves that \eqref{eq:proof:lem:closed_loop_nonzero_moments_2} holds true for all $n \in \llbracket 0, 2N-2 \rrbracket$.
    
    \textbf{Step 2.} \textit{End of the proof.}
    Set $V = \textrm{span}( \varphi_j, j \in J )$. Since $M_n(0) = (-1)^n \mu^{(2n)}$ by \eqref{eq:proof:prop:Lie_brackets_2}, one has $M_n(0) \in V$ for all $n \geq 0$. The assumption $\vert J \vert \leq N-1$ thus implies that
    \begin{equation}
        \sum_{n = 0}^{N-1} \beta_n M_n(0) = 0,
    \end{equation}
    for some $\beta_0, \ldots \beta_{N-1} \in \mathbb{R}$, with at least one $\beta_n$ nonzero. Set $B_1=\sum_{n = 0}^{N-1} \beta_n M_n$, so that $B_1(0) = 0$. Up to replacing $B_1$ by $\ad_{f_0}^m(B_1)$ for some $m \geq 0$, we can assume that $\beta_{N-1} \neq 0$. Define
    \begin{equation}
        B_2 := \left[ B_1, [f_0, B_1] \right] = \left[ \sum_{n = 0}^{N-1} \beta_n M_{n}, \sum_{m = 0}^{N-1} \beta_m M_{m+1} \right] =: \beta_{N-1}^2 W_N + B_3.
    \end{equation}
    On the one hand, since $f_0(0) = B_1(0) = 0$, one has $B_2(0) = 0$. On the other hand, Step 1 gives $\langle B_3(0), \varphi_k \rangle = 0$. It implies $\langle W_N(0), \varphi_k \rangle = 0$, and thus $a_k^N = 0$, a contradiction.
\end{proof}
 
\section{Remainder estimates}\label{sec:remainder_estimates}

\subsection{Linear, quadratic and cubic estimates}

Here, we establish estimates that play a crucial role in the proof of our main results: they allow us to transfer a coercive estimate for the quadratic term $y_2$ into a drift estimate for the full controlled solution. Since we impose strong regularity assumptions on $\mu$ throughout this article, strong remainder estimates are available.

\begin{proposition}\label{prop:quad_cubic_estimate_drift_any_order}
    Let $k \geq 1$, $T>0$, $\mu \in H_{(0)}^4(0,1)$ and $u \in L^2(0,T)$. Recall that $y_1$ and $y_2$ are the solutions of \eqref{eq:Burgers_y_1_intro} and \eqref{eq:Burgers_y_2_intro}, respectively. Set $R_2(t) := y(t; 0, u) - y_1(t;u)$ and $R_3(t) := y(t; 0, u) - y_1(t;u) - y_2(t;u)$. There exists $C > 0$ which depends only on $\mu$ and $k$, such that 
    \begin{align}
        \left\Vert y_1 \right\Vert_{L^2((0,T); H^1)} \leq C \Vert u_1 \Vert_{L^2} \quad \text{ and } \quad
        \left\Vert y_2 \right\Vert_{L^2((0,T); H^1)} \leq C \Vert u_1 \Vert_{L^2}^2. \label{eq:prop:quad_cubic_estimate_drift_any_order_2}
    \end{align}
    In addition, there exists $\delta_0 > 0$ which depends only on $\mu$ and $k$ such that if $\Vert u_1 \Vert_{L^2} \leq \delta_0$, then
    \begin{align}
        \left\Vert R_2 \right\Vert_{L^2((0,T); H^1)} + \left\Vert R_2(T) \right\Vert_{H^{-1}} \leq C \Vert u_1 \Vert_{L^2}^2 \quad \text{ and } \quad
        \left\vert \left\langle R_3(T), \varphi_k \right\rangle \right\vert \leq C \Vert u_1 \Vert_{L^2}^3. \label{eq:prop:quad_cubic_estimate_drift_any_order_4}
    \end{align}
\end{proposition}

\begin{proof}
    In this proof, the symbol $\lesssim$ denotes an inequality with a constant that may depend on $\mu$ and $k$. First, an integration by parts gives
    \begin{equation}\label{eq:proof:prop:quad_cubic_estimate_drift_any_order_1}
        y_1(t) = u_1(t) \mu + \int_0^t u_1(s) \partial_t \Phi (t-s) \dd s,
    \end{equation}
    where $\Phi(t) := S(t) \mu$, with $S$ the heat semigroup. The assumption $\mu \in H_{(0)}^4(0,1)$ is more than sufficient to ensure that $\Vert \partial_t \Phi \Vert_{L^1((0,T);H^1)} \lesssim 1$, yielding
    \begin{equation}\label{eq:proof:prop:quad_cubic_estimate_drift_any_order_2}
        \left\Vert y_1 \right\Vert_{L^2((0,T); H^1)} \leq \Vert u_1 \Vert_{L^2} \Vert \mu \Vert_{H^1} + \Vert u_1 \Vert_{L^2} \Vert \partial_t \Phi \Vert_{L^1((0,T);H^1)} \lesssim \Vert u_1 \Vert_{L^2},
    \end{equation} 
    where we have used Young's inequality. 
    
    Second, Lemma \ref{lem:well-posedness_heat_L1L2} gives 
    \begin{equation}
        \left\Vert y_2 \right\Vert_{L^2((0,T); H^1)} \leq \left\Vert \partial_x \left( (y_1)^2 \right) \right\Vert_{L^1((0,T); L^2)} \lesssim \left\Vert y_1 \right\Vert_{L^2((0,T); H^1)}^2 \lesssim \Vert u_1 \Vert_{L^2}^2,
    \end{equation} 
    by \eqref{eq:proof:prop:quad_cubic_estimate_drift_any_order_2}, and since $H^1(0,1)$ is an algebra. 
    
    Third, note that $R_2$ is the solution of the forced Burgers equation
    \begin{equation}
        \partial_t R_2 - \partial_x^2 R_2 + R_2 \partial_x R_2 + \partial_x \left( y_1 R_2 \right) = - \partial_x \left[ \frac{(y_1)^2}{2} \right].
    \end{equation} 
    Hence, we can apply Lemma \ref{lem:forced_burgers}, with $g := y_1$ and $f := -\frac{(y_1)^2}{2}$. Since 
    \begin{equation}
        \Vert f \Vert_{L^1((0,T); H^1)} \lesssim \left\Vert y_1 \right\Vert_{L^2((0,T); H^1)}^2 \lesssim \Vert u_1 \Vert_{L^2}^2,
    \end{equation}
    and $\Vert g \Vert_{L^2((0,T); H^1)} = \left\Vert y_1 \right\Vert_{L^2((0,T); H^1)} \lesssim \Vert u_1 \Vert_{L^2}$, it proves that 
    \begin{equation}
        \left\Vert R_2 \right\Vert_{L^2((0,T);H^1)} + \left\Vert R_2(T) \right\Vert_{H^{-1}} \lesssim \Vert f \Vert_{L^1((0,T); H^1)} \lesssim \Vert u_1 \Vert_{L^2}^2,
    \end{equation}
    if $\Vert u_1 \Vert_{L^2}$ is sufficiently small.
    
    Fourth, $R_3$ is the solution of the forced Burgers equation
    \begin{equation}
        \partial_t R_3 - \partial_x^2 R_3 + R_3 \partial_x R_3 + \partial_x \left[ \left( y_1 + y_2 \right) R_3 \right] = - \partial_x \left[ y_1 y_2 + \frac{(y_2)^2}{2} \right].
    \end{equation} 
    Hence, we apply Lemma \ref{lem:forced_burgers}, this time with $g := y_1+y_2$ and $f := - y_1 y_2 - \frac{(y_2)^2}{2}$. As above, one has 
    \begin{equation}
        \Vert f \Vert_{L^1((0,T); H^1)} \lesssim \Vert u_1 \Vert_{L^2}^3 + \Vert u_1 \Vert_{L^2}^4    
    \end{equation}
     and $\Vert g \Vert_{L^2((0,T); H^1)} \lesssim \Vert u_1 \Vert_{L^2}+\Vert u_1 \Vert_{L^2}^2$, implying 
    \begin{equation}
        \left\vert \left\langle R_3(T), \varphi_k \right\rangle \right\vert \lesssim \Vert f \Vert_{L^1(L^1)} + \left( \Vert g \Vert_{L^2(L^2)} + \left\Vert f \right\Vert_{L^1(L^2)} \right) \left\Vert f \right\Vert_{L^1(L^2)} \lesssim \Vert u_1 \Vert_{L^2}^3,
    \end{equation}
    if $\Vert u_1 \Vert_{L^2}$ is sufficiently small.
\end{proof}
 
\subsection{Closed-loop estimates}

We prove the following closed-loop estimates. Recall that $u_1$, $u_2$, $\ldots$ denote the iterated primitives of $u \in L^2(0,T)$.

\begin{lemma}\label{lem:closed_loop}
    Let $\mu \in H^{-1}(0,1)$, and $N \geq 1$. Assume that there exists $J \subset \mathbb{N}^\ast$ such that $\vert J \vert = N$ and $\langle \mu, \varphi_j \rangle \neq 0$ for all $j \in J$. Then, there exists $C > 0$ such that for all $T \in \left( 0, 1 \right)$ and $u \in L^2(0,T)$, one has
    \begin{equation}\label{eq:lem:closed_loop_2}
        \sum_{n = 1}^N \left\vert u_n(T) \right\vert \leq C \left( \sqrt{T} \left\Vert u_N \right\Vert_{L^2} +  \left\Vert y_1(T) \right\Vert_{H^{-1}} \right). 
    \end{equation}
\end{lemma}

\begin{proof}
    For $j \geq 1$, starting from the definition of $y_1(T)$ and integrating by parts, one finds
    \begin{equation}
        \begin{split}
            \left\langle y_1(T), \varphi_j \right\rangle 
             & = \left\langle \mu, \varphi_j \right\rangle \left( \int_0^T e^{- \lambda_j (T-t)} u(t) \dd t \right) \\
             & = \left\langle \mu, \varphi_j \right\rangle \left( \sum_{n = 1}^N (-\lambda_j)^{n-1} u_n(T) + (-\lambda_j)^N \int_0^T e^{- \lambda_j (T-t)} u_N(t) \dd t \right).
        \end{split}
    \end{equation}
    For $j \in J$, using the Cauchy-Schwarz inequality again, this gives
    \begin{equation}
        \begin{split}
            \left\vert \sum_{n = 1}^N (-\lambda_j)^{n-1} u_n(T) \right\vert 
            & \leq C \left( \left\vert \int_0^T e^{- \lambda_j (T-t)} u_N(t) \dd t \right\vert + \left\vert \left\langle y_1(T), \varphi_j \right\rangle \right\vert \right) \\
            & \leq C^\prime \left( \sqrt{T} \left\Vert u_N \right\Vert_{L^2} + \left\Vert y_1(T) \right\Vert_{H^{-1}} \right),
        \end{split}
    \end{equation}
    with $C, C^\prime > 0$ depending on $\mu$ and $N$. Using the invertibility of the Vandermonde matrix of coefficients $(-\lambda_j)_{j \in J}$, one obtains \eqref{eq:lem:closed_loop_2}.
\end{proof}

\subsection{Decoupling estimate for initial data and control}

We now provide the decoupling estimates used in the proof of our main result. The proof is an elementary consequence of Lemma \ref{lem:forced_burgers} and of standard well-posedness results for the Burgers equation; we omit it and refer to \cite{BurgersFractional}. 

\begin{lemma}\label{lem:decoupling_estimate}
    Let $T \in (0,1]$, $y_0 \in L^2(0,1)$, $k \geq 1$, and $\mu \in H^{-1}(0,1)$. Let $u$ be such that the solutions $y(t; y_0, u)$ and $y(t; 0, u)$ of the Burgers equation \eqref{eq:Burgers_intro} are well-defined and belong to $L^2((0,T); H^1(0,1))$. Set $R(t; y_0, u) :=  y(t; y_0, u) - y(t; y_0, 0) - y(t; 0, u)$, where the solutions of the Burgers equation implicitly depend on $\mu$. We introduce the shorthand
    \begin{equation}
        y^u(t) :=  y(t; 0, u).
    \end{equation}
    Then, there exist constants $\varepsilon_0, C > 0$, which may depend only on $k$, such that if 
    \begin{equation}\label{eq:lem:decoupling_estimate_0}
        \Vert y_0 \Vert_{L^2} + \left\Vert y^u \right\Vert_{L^2((0, T);H^1)} \leq \varepsilon_0, 
    \end{equation}
    then
    \begin{equation}\label{eq:lem:decoupling_estimate_1}
        \left\Vert R(T; y_0, u) \right\Vert_{H^{-1}}^2 \leq C \sqrt{T} \Vert y_0 \Vert_{L^2}^2 \left\Vert y^u \right\Vert_{L^2((0,T);L^2)}^2,
    \end{equation}
    and 
    \begin{equation}\label{eq:lem:decoupling_estimate_2}
        \left\vert \left\langle R(T; y_0, u), \varphi_k \right\rangle \right\vert \leq C T^{\frac{1}{4}} \Vert y_0 \Vert_{L^2} \left\Vert y^u \right\Vert_{L^2((0,T);L^2)}.
    \end{equation}
\end{lemma}

\section{Obstructions to controllability}\label{sec:proof_obstructions}

\subsection{Drift estimates}

We now prove Theorem \ref{thm:drift_any_order}. 
Let $N \geq 2$, $k \geq 1$, and $\mu \in H_{(0)}^{4N}(0,1)$ be such that $a_k^1(\mu) = \cdots = a_k^{N-1}(\mu) = 0$, $a_k^N(\mu) \neq 0$, and $\langle \mu, \varphi_k \rangle = 0$. In this proof, the symbol $\lesssim$ is used for constants that may depend on $\mu$, $N$ and $k$. We write $a_k^N$ instead of $a_k^N(\mu)$.

Let $y_0 \in L^2(0, 1)$, $T \in (0,1)$ and $u \in L^2(0,T)$ with $\Vert u \Vert_{L^2} \leq 1$. 
Write $t \mapsto R_3(t; u) := y(t; 0, u) - y_1(t; u) - y_2(t; u)$ for the cubic remainder, and $t \mapsto R(t; y_0, u) := y(t; y_0, u) - y(t; y_0, 0) - y(t; 0, u)$ for the remainder of the decoupling estimate (Lemma \ref{lem:decoupling_estimate}). Using $\langle \mu, \varphi_k \rangle = 0$ and the quadratic coercive estimate (Proposition \ref{prop:quad_estimates_IPP}), one finds 
\begin{equation}
    \begin{split}
         & \left\vert \left\langle y(T; y_0, u) - y(T; y_0, 0), \varphi_k \right\rangle - a_k^N \Vert u_N \Vert_{L^2}^2 \right\vert \\
        \leq \ & \left\vert \left\langle y_2(T; u), \varphi_k \right\rangle - a_k^N \Vert u_N \Vert_{L^2}^2 \right\vert 
         + \left\vert \left\langle R_3(T;u), \varphi_k \right\rangle \right\vert + \left\vert \left\langle R(T;y_0,u), \varphi_k \right\rangle \right\vert \\
        \lesssim \ & T \Vert u_N \Vert_{L^2}^2 + \left\vert u_1(T) \right\vert^2 + \cdots + \left\vert u_N(T) \right\vert^2 
         + \left\vert \left\langle R_3(T;u), \varphi_k \right\rangle \right\vert + \left\vert \left\langle R(T;y_0,u), \varphi_k \right\rangle \right\vert.
    \end{split}
\end{equation}
By Lemma \ref{lem:closed_loop_nonzero_moments}, there exists $J \subset \mathbb{N}^\ast$ such that $\langle \mu, \varphi_j \rangle \neq 0$ for all $j \in J$, and $\vert J \vert \geq N$, allowing us to use closed-loop estimates (Lemma \ref{lem:closed_loop}), which yield
\begin{equation}
    \left\vert u_1(T) \right\vert^2 + \cdots + \left\vert u_N(T) \right\vert^2 
    \lesssim T \left\Vert u_N \right\Vert_{L^2}^2 +  \left\Vert y_1(T;u) \right\Vert_{H^{-1}}^2. 
\end{equation}
Writing $y_1$ in terms of $y$, $R$ and $R_2$, one obtains
\begin{equation}
    \begin{split}
         & \left\vert \left\langle y(T; y_0, u) - y(T; y_0, 0), \varphi_k \right\rangle - a_k^N \Vert u_N \Vert_{L^2}^2 \right\vert \\
        \lesssim \ & T \Vert u_N \Vert_{L^2}^2 + \left\Vert  y(T; y_0, u) - y(T; y_0, 0) \right\Vert_{H^{-1}}^2 + \left\Vert R_2(T;u) \right\Vert_{H^{-1}}^2 + \left\Vert R(T;y_0,u) \right\Vert_{H^{-1}}^2 \\ 
        & \quad + \left\vert \left\langle R_3(T;u), \varphi_k \right\rangle \right\vert + \left\vert \left\langle R(T;y_0,u), \varphi_k \right\rangle \right\vert.
    \end{split}
\end{equation}
Now, we use the remainder estimates (Proposition \ref{prop:quad_cubic_estimate_drift_any_order}) and the decoupling estimate (Lemma \ref{lem:decoupling_estimate}). 

First, by Proposition \ref{prop:quad_cubic_estimate_drift_any_order}, there exists $\delta_0 > 0$ such that if $\Vert u_1 \Vert_{L^2} \leq \delta_0$, then 
\begin{equation}
    \left\Vert R_2(T;u) \right\Vert_{H^{-1}}^2 + \left\vert \left\langle R_3(T;u), \varphi_k \right\rangle \right\vert \lesssim \Vert u_1 \Vert_{L^2}^3.
\end{equation}
Note that 
\begin{equation}\label{eq:proof:thm:drift_any_order_1}
    \Vert u_1 \Vert_{L^2}^2 \leq \int_0^T t \left( \int_0^t u(s)^2 \dd s \right) \dd t \leq \frac{T^2}{2} \Vert u \Vert_{L^2}^2 \leq \frac{T^2}{2},
\end{equation}
so that the condition $\Vert u_1 \Vert_{L^2} \leq \delta_0$ is implied by $T \leq \frac{\delta_0}{\sqrt{2}}$. 

Second, by Lemma \ref{lem:decoupling_estimate}, there exists $\varepsilon_0$ such that if $\Vert y_0 \Vert_{L^2} + \left\Vert y^u \right\Vert_{L^2(H^1)} \leq \varepsilon_0$, where $y^u(t) := y(t; 0, u)$, then
\begin{equation}
    \left\Vert R(T;y_0,u) \right\Vert_{H^{-1}}^2 + \left\vert \left\langle R(T;y_0,u), \varphi_k \right\rangle \right\vert \lesssim T^{\frac{1}{4}} \Vert y_0 \Vert_{L^2} \left\Vert y^u \right\Vert_{L^2(L^2)}.
\end{equation}
Since Proposition \ref{prop:quad_cubic_estimate_drift_any_order} implies 
\begin{equation}
    \Vert y^u \Vert_{L^2(H^1)} \leq \Vert y_1 \Vert_{L^2(H^1)} + \Vert R_2 \Vert_{L^2(H^1)} \lesssim \Vert u_1 \Vert_{L^2},
\end{equation}
there exists $\delta_1 > 0$ such that if $\Vert y_0 \Vert_{L^2} \leq \frac{\varepsilon_0}{2}$ and $\Vert u_1 \Vert_{L^2} \leq \delta_1$, then 
\begin{equation}
    \left\Vert R(T;y_0,u) \right\Vert_{H^{-1}}^2 + \left\vert \left\langle R(T;y_0,u), \varphi_k \right\rangle \right\vert \lesssim T^{\frac{1}{4}} \Vert y_0 \Vert_{L^2} \Vert u_1 \Vert_{L^2}.
\end{equation}

Summarizing, we have proved that if $T \leq \frac{\min(\delta_0, \delta_1)}{\sqrt{2}}$ and $\Vert y_0 \Vert_{L^2} \leq \frac{\varepsilon_0}{2}$, then
\begin{equation}
    \begin{split}
         & \left\vert \left\langle y(T; y_0, u) - y(T; y_0, 0), \varphi_k \right\rangle - a_k^N \Vert u_N \Vert_{L^2}^2 \right\vert \\
        \lesssim \ & T \Vert u_N \Vert_{L^2}^2 +  \left\Vert  y(T; y_0, u) - y(T; y_0, 0) \right\Vert_{H^{-1}}^2  + \Vert u_1 \Vert_{L^2}^3 + T^{\frac{1}{4}} \Vert y_0 \Vert_{L^2} \Vert u_1 \Vert_{L^2}.
    \end{split}
\end{equation} 
By the Gagliardo--Nirenberg inequality (see \cite{Gagliardo1959,Nirenberg1959}), one has
\begin{equation}
    \begin{split}
        \Vert u_1 \Vert_{L^2}^3 = \left\Vert u_N^{(N-1)} \right\Vert_{L^2}^3
        & \lesssim \left\Vert u_N^{(3N-3)} \right\Vert_{L^2} \Vert u_N \Vert_{L^2}^2 + T^{-3(N-1)} \Vert u_N \Vert_{L^2}^3  \\
        & \lesssim \left( \left\Vert u^{(2N-3)} \right\Vert_{L^2} + T^{-2N+3} \Vert u \Vert_{L^2} \right) \Vert u_N \Vert_{L^2}^2,
    \end{split}
\end{equation}
and this completes the proof of Theorem \ref{thm:drift_any_order}. 

\subsection{Proof of obstructions}

Here we prove Corollary \ref{cor:obstruction_small_time}. Again, the symbol $\lesssim$ is used for constants that may depend on $\mu$, $N$, and $k$. Using Theorem \ref{thm:drift_any_order} and the estimate for the free Burgers equation proved in Lemma \ref{lem:free_evolution_Burgers}, one finds
\begin{equation}
    \begin{split}
        & \left\vert \left\langle y(T; y_0, u) - y_0, \varphi_k \right\rangle - a_k^N \Vert u_N \Vert_{L^2}^2 \right\vert \\
        \lesssim \, & \Gamma_N(T,u) \Vert u_N \Vert_{L^2}^2  
         + \Vert y(T; y_0, u) - y(T; y_0, 0) \Vert_{H^{-1}}^2 + T^{\frac{1}{4}} \Vert y_0 \Vert_{L^2} \Vert u_1 \Vert_{L^2} + T^{\frac{1}{2}} \Vert y_0 \Vert_{L^2} \\
        \lesssim \, & \Gamma_N(T,u) \Vert u_N \Vert_{L^2}^2 + \Vert y(T; y_0, u) \Vert_{H^{-1}}^2 + T^{\frac{1}{4}} \Vert y_0 \Vert_{L^2} + \Vert y_0 \Vert_{L^2}^2,
    \end{split}
\end{equation}
where we have also used $\Vert y_0 \Vert_{L^2} \lesssim 1$, $T \lesssim 1$ and $\Vert u \Vert_{L^2} \leq 1$. We assume that $a_k^N > 0$; the case $a_k^N < 0$ is similar. Since $\langle \psi, \varphi_k \rangle \lesssim \Vert \psi \Vert_{H^{-1}}$, one finds
\begin{equation}
    \begin{split}
        & a_k^N \Vert u_N \Vert_{L^2}^2 + \left\langle  y_0, \varphi_k \right\rangle \\
        \lesssim \ & \Gamma_N(T,u) \Vert u_N \Vert_{L^2}^2 + \Vert y(T; y_0, u) \Vert_{H^{-1}}^2 + \Vert y(T; y_0, u) \Vert_{H^{-1}} + T^{\frac{1}{4}} \Vert y_0 \Vert_{L^2} + \Vert y_0 \Vert_{L^2}^2.
    \end{split}
\end{equation}
Hence, if $y_0 = \eta \varphi_k$, then there exists $C > 0$ such that 
\begin{equation}
    \begin{split}
        & \left( a_k^N - C \Gamma_N(T,u) \right) \Vert u_N \Vert_{L^2}^2 + \left( 1 - C T^{\frac{1}{4}} - C \eta \right) \eta  \\
        \leq \ & C \left\Vert y(T; y_0, u) \right\Vert_{H^{-1}} + C \left\Vert y(T; y_0, u) \right\Vert_{H^{-1}}^2.
    \end{split}
\end{equation}
Assume that $\eta \leq \frac{1}{4C} =: \eta_1$. Choose $T_1 > 0$ such that $C T_1 \leq \frac{a_k^N}{4}$, and $C T_1^{\frac{1}{4}} \leq \frac{1}{4}$. Then 
\begin{equation}
    \begin{split}
        C \Gamma_N(T,u)
        & \leq \frac{a_k^N}{4} + C \left\Vert u^{(2N-3)} \right\Vert_{L^2} + C T^{-2N+3} \Vert u \Vert_{L^2} \\
        & \leq \frac{a_k^N}{4} + C \left( 1 + T^{-2N+3} \right) \left\Vert u \right\Vert_{H^{2N-3}} .
    \end{split}
\end{equation}
Hence, there exists $\delta_1 = \delta_1(T) > 0$ such that if $\left\Vert u \right\Vert_{H^{2N-3}} \leq \delta_1$, then $C \Gamma_N(T,u) \leq \frac{a_k^N}{2}$, and so
\begin{equation}
    \begin{split}
        &  \frac{a_k^N}{2} \Vert u_N \Vert_{L^2}^2 + \frac{\eta}{2}  
        \leq  C \left\Vert y(T; y_0, u) \right\Vert_{H^{-1}} + C \left\Vert y(T; y_0, u) \right\Vert_{H^{-1}}^2.
    \end{split}
\end{equation}
This completes the proof of Corollary \ref{cor:obstruction_small_time}.

\section{Construction of profiles satisfying the drift conditions}\label{sec:proof_existence_mu}

Here, we prove Proposition \ref{prop:existence_mu}. Recall that $a_k^n(\mu)$ is defined by \eqref{eq:def:a_k_n}. Let $k \geq 1$ and $N \geq 2$. 
    
\textbf{Step 1:} \textit{choice of the form of $\mu$.} We construct $\mu$ supported on finitely many frequencies, that is, $\mu = \sum_{j \in J} \mu_j \varphi_j$, with $J \subset \mathbb{N}^\ast$ and $\vert J \vert < + \infty$. Note that $a_k^n(\mu)$ is a linear combination of terms of the form $\left\langle \mu^{(2a)} \mu^{(2b)} , \varphi_k^\prime \right\rangle$. Let $a, b \geq 0$. One has
\begin{equation}
    \left\langle \mu^{(2a)} \mu^{(2b)} , \varphi_k^\prime \right\rangle
    = \frac{k \pi}{\sqrt{2}} \sum_{p,q \in J} \mu_p \mu_q \left(-\lambda_p \right)^a \left(-\lambda_q \right)^b \left( \mathbbm{1}_{\vert p - q \vert = k} - \mathbbm{1}_{p + q = k} \right).
\end{equation}
The factor $\mathbbm{1}_{\vert p - q \vert = k} - \mathbbm{1}_{p + q = k}$ motivates the choice of $J$ of the form $\left\{ m_1, m_1 + k, \ldots, m_M, m_M + k \right\}$, for some integers $m_1 < \cdots < m_M$, and $M \geq 2$. Assuming that $m_1 > k$, and that $m_{j+1} - (m_{j} + k) > k$ for all $j \in \llbracket 1, M-1 \rrbracket$, this choice yields
\begin{equation}
    \left\langle \mu^{(2a)} \mu^{(2b)} , \varphi_k^\prime \right\rangle
    = \frac{k \pi}{\sqrt{2}} \sum_{j=1}^M \mu_{m_j} \mu_{m_j + k} \left( \left(-\lambda_{m_j} \right)^a \left(-\lambda_{m_j + k} \right)^b + \left(-\lambda_{m_j+k} \right)^a \left(-\lambda_{m_j} \right)^b \right). 
\end{equation}
Note that $m_1 > k$ has been used to remove $\mathbbm{1}_{p + q = k}$, and that it implies $\langle \mu, \varphi_k \rangle = 0$. Only the products $\mu_{m_j}\mu_{m_j+k}$ appear in the preceding expression; we may therefore choose $\mu_{m_j+k}=1$ and write $\nu_j:=\mu_{m_j}$, for $j\in\llbracket 1, M\rrbracket$. To simplify the computations, we choose each $m_j$ to be a multiple of $k$, that is, $m_j = k n_j$, for some $n_j \in \mathbb{N}^\ast$. To ensure that the preceding conditions remain satisfied, the parameters $n_j$ must be chosen such that $n_1 > 1$ and $n_{j+1} - n_{j} > 2$ for all $j \in \llbracket 1, M-1 \rrbracket$. Since $\lambda_{m_j} = (k\pi)^2 n_j^2$ and $\lambda_{m_j+k}= (k\pi)^2 (n_j+1)^2$, we finally obtain
\begin{equation}
    \left\langle \mu^{(2a)} \mu^{(2b)} , \varphi_k^\prime \right\rangle
    = \frac{\left( k \pi \right)^{2(a+b)+1}}{\sqrt{2}} (-1)^{a+b} \sum_{j=1}^M \nu_j P_{a,b}(n_j),
\end{equation}
where $P_{a,b}(X) := X^{2a} (X+1)^{2b} + (X+1)^{2a} X^{2b}$, with $\mu$ of the form $\sum_{j = 1}^M \left( \nu_j \varphi_{k n_j} + \varphi_{k n_j + k} \right)$.

\textbf{Step 2:} \textit{choice of the parameters.} 
A straightforward calculation gives $a_k^n(\mu) = \sum_{j = 1}^M \nu_j Q_n(n_j)$,  with 
\begin{equation}
    Q_n(X) := \frac{(k\pi)^{4n-3}}{2 \sqrt{2}}  \sum_{\alpha = 0}^{n - 1} \sum_{\beta = 0}^{n-1-\alpha} (-1)^{\alpha} \Lucas{n - 1 - \beta}{\alpha} P_{\beta, n-1+\alpha}(X).
\end{equation}
We choose $M = N$. We claim that the parameters $n_1, \cdots, n_N$ can be chosen so that the matrix $\mathcal{M} := \left( Q_n(n_j) \right)_{1 \leq n,j \leq N}$ is invertible. Admitting that this is true, we can simply choose 
\begin{equation}
    \left( \nu_1, \ldots, \nu_N \right)^{\intercal} := \mathcal{M}^{-1} \left( 0, \ldots, 0, 1 \right)^{\intercal},
\end{equation}
to obtain $a_k^1(\mu) = \cdots = a_k^{N-1}(\mu) = 0$ and $a_k^N(\mu) = 1$.

It remains to prove the claim. For $n \geq 1$, the polynomial $Q_n$ is of degree at most $4n-4$, and the coefficient in front of $X^{4n - 4}$ is
\begin{equation}
    \frac{(k\pi)^{4n-3}}{2 \sqrt{2}} \sum_{\alpha = 0}^{n - 1} (-1)^{\alpha} \Lucas{\alpha}{\alpha} \times 2 = \frac{(k\pi)^{4n-3}}{\sqrt{2}} \left( 1 + 2 \sum_{\alpha = 1}^{n-1} (-1)^\alpha \right) =  \frac{(k\pi)^{4n-3} (-1)^{n-1}}{\sqrt{2}},
\end{equation}
implying that $Q_n$ is precisely of degree $4 n - 4$. In particular, $(Q_1, \ldots, Q_N)$ are linearly independent. To ensure that $n_{j+1} - n_{j} > 2$, we simply choose the integers $n_j$ as multiples of $3$. Set 
\begin{equation}
    E := \mathrm{span} \left\{ \left( Q_1(n), \ldots, Q_N(n) \right), \  n \in 3 \mathbb{N}^\ast \right\} \subset \mathbb{R}^N.
\end{equation}
Assume by contradiction that $E \neq \mathbb{R}^N$. Then, there exists a nonzero vector $(a_1,\ldots,a_N) \in \mathbb{R}^N$ in the orthogonal complement of $E$ with respect to the standard scalar product on $\mathbb{R}^N$, yielding $\sum_{j = 1}^N a_j Q_j(n) = 0$ for all $n \in 3 \mathbb{N}^\ast$. Thus, the polynomial $\sum_{j=1}^N a_j Q_j$ has infinitely many roots and must therefore vanish identically, which contradicts the linear independence of $Q_1,\ldots,Q_N$. Hence, $E=\mathbb{R}^N$. In particular, there exist $n_1,\ldots,n_N \in 3\mathbb{N}^\ast$, with $n_1<\cdots<n_N$, such that the matrix $\left(Q_n(n_j)\right)_{1 \leq n,j \leq N}$ is invertible.

\appendix

\section{Regularity results for heat and Burgers equations}\label{sec:appendix_heat_burgers}

First, we recall, without proof, the following classical result. The fact that the continuity constant is independent of $T$ can be verified, for instance, by a straightforward inspection of \cite[Lemma A.1]{Nguyen2025Burgers}. 

\begin{lemma}\label{lem:well-posedness_heat_L1L2}
    If $f \in L^1((0, T); L^2(0, 1))$ and $y_0 \in L^2(0, 1)$, then there exists a unique solution $y$ of 
    \begin{equation}\label{eq:lem:heat_weak_wellposedness}
        \left\{
        \begin{array}{cll}
            \partial_t y - \partial_x^2 y = f
            & \quad t \in (0, T), & \quad x \in (0,1), \\
            y(t,0) = y(t,1) = 0
            & & \quad t \in (0, T), \\
            y(0,x)=y_0
            & & \quad x \in (0, 1).
        \end{array}
        \right.
    \end{equation}
     satisfying $y \in C^0([0, T]; L^2(0,1)) \cap L^2((0, T); H_0^1(0,1))$. In addition, there exists $C > 0$ independent of $T$, such that
    \begin{equation}\label{eq:lem:well-posedness_1}
        \Vert y \Vert_{L^\infty((0, T);L^2)} + \Vert y \Vert_{L^2((0, T);H_0^1)} \leq C \left( \Vert y_0 \Vert_{L^2} + \Vert f \Vert_{L^1((0, T); L^2)} \right).
    \end{equation}
\end{lemma} 

Second, the following result is used repeatedly to prove remainder estimates. The proof is elementary and we omit it (see, for instance, \cite{BurgersFractional}).

\begin{lemma}[Forced Burgers equations]\label{lem:forced_burgers}
    Let $T > 0$, and $k \geq 1$. Let $f \in L^1((0,T); H^1(0,1))$ and $g \in L^2((0,T); H^1(0,1))$. There exist $C, \delta > 0$, which may depend only on $k$, such that if $\Vert f \Vert_{L^1((0,T);H^1)} \leq \delta$ and $\Vert g \Vert_{L^2((0,T);H^1)} \leq \delta$, then the solution $y$ of 
    \begin{equation}\label{eq:forced_Burgers}
        \left\{
        \begin{array}{lll}
            \partial_t y - \partial_x^2 y + y \partial_x y + \partial_x (g y) = \partial_x f
            & \quad t \in (0, T), & \quad x \in (0,1), \\
            y(t,x) = 0
            & \quad t \in (0, T),  & \quad x \in \{0, 1\},\\
            y(0,x)=0
            & & \quad x \in (0, 1),
        \end{array}
        \right.
    \end{equation}
    satisfies
    \begin{equation}\label{eq:lem:forced_burgers_1}
        \left\Vert y \right\Vert_{L^2((0,T);H^1)} \leq C \Vert f \Vert_{L^1(H^1)},
    \end{equation}
    \begin{equation}\label{eq:lem:forced_burgers_2}
        \left\Vert y \right\Vert_{L^2((0,T);L^2)} + \left\Vert y(T) \right\Vert_{H^{-1}} \leq C \left\Vert f \right\Vert_{L^1((0,T);L^2)},
    \end{equation}
    and 
    \begin{equation}\label{eq:lem:forced_burgers_3}
        \begin{split}
            \left\vert \left\langle y(T), \varphi_k \right\rangle \right\vert
            \leq \ & C \left\Vert f \right\Vert_{L^1((0,T);L^1)} + C \left( \Vert g \Vert_{L^2((0,T);L^2)} + \left\Vert f \right\Vert_{L^1((0,T);L^2)} \right) \left\Vert f \right\Vert_{L^1((0,T);L^2)}.
        \end{split}
    \end{equation}
\end{lemma}

Finally, we need the following estimate on the free evolution of a Burgers equation with small initial data. Again, the proof is elementary and we omit it (see, for instance, \cite{BurgersFractional}).

\begin{lemma}\label{lem:free_evolution_Burgers}
    There exists $C > 0$ such that for all $T \in (0, 1]$ and $y_0 \in L^2(0, 1)$, with $\Vert y_0 \Vert_{L^2} \leq 1$, one has 
    \begin{equation}\label{eq:lem:free_evolution_Burgers_2}
        \left\Vert y(T; y_0, 0) - y_0 \right\Vert_{H^{-1}} \leq C T^{\frac{1}{2}} \Vert y_0 \Vert_{L^2}.
    \end{equation}
\end{lemma}

\section*{Acknowledgments}

The author acknowledges support from the Fondation Simone et Cino Del Duca - Institut de France, the Fondation Université de Rennes and from grant ANR-11-LABX-0020 (Labex Lebesgue). The author is deeply grateful to Karine Beauchard and Frédéric Marbach for introducing him to the subject and for their constant support throughout this work.

\bibliographystyle{plain}
\bibliography{biblio}

\noindent
\textsc{Thomas Perrin:} \texttt{thomas.perrin@ens-rennes.fr}

\end{document}